\documentclass[12pt]{amsart}
\usepackage{amssymb,amsmath}
\usepackage{tikz}
\usepackage{tikz-cd}
\usepackage{float}

\def\dlim{\displaystyle\lim}
\def\dint{\displaystyle\int}
\def\dsum{\displaystyle\sum}
\def\dprod{\displaystyle\prod}
\def\t{\mathbf{t}}
\def\x{\mathbf{x}}
\def\y{\bar{\mathbf{y}}}
\def\M{\bar{\mathcal{M}}}
\def\H{\mathcal{H}}
\def\<{\left<}
\def\>{\right>}
\def\S{\mathbf{\mathcal{S}}}
\DeclareMathOperator{\im}{Im}

\newtheorem{thm}{Theorem}

\newtheorem{lemma}{Lemma}[section]
\newtheorem{lemma*}{Lemma}

\newtheorem{prop}{Proposition}[section]
\newtheorem{cor}{Corollary}[section]

\begin{document}

\title{Reconstruction of $g=1$ permutation equivariant quantum $K$-invariants}

\author{Dun Tang}

\begin{abstract}
In this paper, we establish an analog of Dijkgraaf-Witten's theorem for $g=1$ invariants in permutation-equivariant quantum K-theory. This result generalizes the findings of \cite{Tang1} and \cite{Tang2}.
\end{abstract}

\maketitle

\tableofcontents

\section*{Introduction}
\

In \cite{Tang1}\cite{Tang2}, we proved two reconstruction theorems for $g=1$ descendant (i.e., containing the universal cotangent bundle $L$) quantum K-invariants, in terms of $g=0$ invariants and $g=1$ invariants with only $1$ point carrying the universal cotangent bundle, as analogs of Dijkgraaf-Witten's theorem.
More precisely, in \cite{Tang1}, we proved such a theorem for non-permutative quantum $K$-theory.
In \cite{Tang2} Section 4, we proved such a theorem for permutation equivariant $g=1$ invariants, but only for the point target space.
In this paper, we prove the most general form of such a theorem -- the reconstruction theorem for permutation equivariant invariants with general K\"ahler targets.

Other known results for $g=1$ quantum K-theory are in the non-permutable theory.
From the computational point of view, Y.P. Lee and F. Qu gave an effective algorithm that computes $\mathcal{F}_1(\t)$ for the point target space \cite{Lee12}.
From a theoretical point of view, Y.C. Chou, L. Herr and Y.P. Lee \cite{Chou23} constructed genus $g$ quantum K invariants of a general target $X$, from permutation equivariant $g=0$ invariants of the orbifold $Sym^{g+1}X = X^{g+1}/S_{g+1}$, where $S_{g+1}$ is the permutation group of $g+1$ elements, that permutes the $g+1$ copies of $X$.

\subsection{Definition of invariants}
\

Let $X$ be a compact K\"ahler manifold.
Let $\Lambda$ be a local $\lambda$-algebra that contains Novikov's variables and let $K=K^0(X) \otimes \Lambda$.
The permutation equivariant K-theoretic Gromov-Witten invariants take as input $\mathbf{t}_r \in \mathcal{K}_+ = K[q^\pm]$ for each $r \in \mathbb{Z}_+$, and take values in $\Lambda$.

We define these invariants as follows \cite{perm9}.
The moduli space $\M_{g,n}(X,d)$ of degree $d$ stable maps from genus $g$ nodal curves with $n$ marked points to $X$ carries a virtual structural sheaf $\mathcal{O}^{vir}_{g,n,d}$ \cite{Lee01}, which is equivariant under the $S_n$ action on $\M_{g,n}(X,d)$ given by renumbering marked points.
At the $i^\text{th}$ marked point, there is the universal cotangent bundle $L_i$, and an evaluation map $ev_i: \M_{g,n}(X,d) \to X$.
Given a permutation $h \in S_n$ with $\ell_r$ $r$-cycles and inputs $\mathbf{t}_{r,k}(q) \in K[q^\pm]$ for each $r$-cycle $k = (k_1,\cdots,k_r)$, we set $T_{r,k} = \otimes_{i=1}^r \bigl(ev^\ast_{k_i}(\mathbf{t}_{r,k})\bigr)\bigl(L_{k_i}\bigr)$.
Then $T_{r,k}$ is an $h$-(orbi-)bundle over $\M_{g,n}(X,d)$, with $h$ acting as $k$, permuting the components in the tensor product.
We define the correlators
\[\left<\mathbf{t}_{1,1}, \cdots, \mathbf{t}_{1,\ell_1}; \cdots , \mathbf{t}_{r,k}, \cdots \right>_{g,\ell,d} 
= \Bigl(\prod_{r=1}^n r^{-\ell_r}\Bigr) \, str_h \, H^\ast\!\bigl(\M_{g,n}(X,d); \,\mathcal{O}^{vir}_{g,n,d} \otimes \bigl(\otimes_{r,k} T_{r,k}\bigr)\bigr).\]
Note that the super-trace on the right-hand side depends only on the type $\ell=(\ell_1,\cdots,\ell_n)$ of $h$.
Thus, it is well-defined.

\subsection{Ancestor-Descendant correspondence in $g=1$}
\

Our primary tool is the permutation equivariant Ancestor-Descendant correspondence \cite{perm7}\cite{Tang2}.

The inputs of the generating functions $\mathcal{F}_1, \bar{\mathcal{F}}_1$ are $\mathbf{t}=(\mathbf{t}_1, \cdots,\mathbf{t}_r, \cdots)$ and $\bar{\mathbf{t}}=(\bar{\mathbf{t}}_1, \cdots,\bar{\mathbf{t}}_r, \cdots), \tau = (\tau_1,\cdots, \tau_r,\cdots)$, with each $\t_r, \bar{\t}_r \in K[q^\pm]$ and $\tau_r \in K$.
We define $\ell!=\prod_r \ell_r!$, $|\ell| = \sum_r r \ell_r$, and generating functions
\[\begin{array}{ll}
\mathcal{F}_1(\mathbf{t}) &= \dsum_{\ell,d}\frac{Q^d}{\ell!}\left<\mathbf{t}_1(L), \cdots, \mathbf{t}_1(L); \cdots , \mathbf{t}_r(L), \cdots \right>_{1,\ell,d},\\
\bar{\mathcal{F}}_1(\bar{\mathbf{t}}) &= \dsum_{d,\ell,\ell_\tau} \frac{Q^d}{\ell!\ell_\tau!}\left<\bar{\mathbf{t}}_1(\bar{L}), \cdots, \bar{\mathbf{t}}_1(\bar{L}), \tau_1,\cdots,\tau_1;\cdots, \bar{\mathbf{t}}_r(\bar{L}), \cdots, \tau_r,\cdots \right>_{1,\ell+\bar{\ell},d},
\end{array}\]
here $\bar{L}$ stands for the counterpart of $L$, obtained by pulling back $L$'s (on $\M_{1,|\ell|}$) along the forgetful map $\M_{1,|\ell+\bar{\ell}|}(X,d) \to \M_{1,|\ell|}$.

Given inputs $A_{1,1},\cdots,A_{1,\ell_1},\cdots,A_{k,i}, \cdots \in K[q^{\pm}]$ and $\tau=(\tau_1,\cdots) \in K^\infty$, where $A_{k,i}$'s are inputs for $k$-cycles, we define generating functions
\[\left<\!\left<A_{1,1},\cdots,A_{1,\ell_1},\cdots,A_{k,i}, \cdots \right>\!\right>_{g,\ell} = \sum_{\bar{\ell},d} \frac{Q^d}{\bar{\ell}!} \left<A_{1,1},\cdots,A_{1,\ell_1},\cdots,A_{k,i}, \cdots;\cdots,\tau_k, \cdots\right>_{g,\ell+\bar{\ell},d}.\]

Pick a basis $\phi_\alpha$ of $K^0(X)$, and set
\[\begin{array}{ll}
(\varphi, \psi) &= \chi(X,\varphi \otimes \psi),\\
G_{\alpha\beta} &= (\phi_\alpha,\phi_\beta)+\left<\!\left< \phi_\alpha, \phi_\beta\right>\!\right>_{0,2_1}.
\end{array}\]
Here $2_1$ stands for two $1$-cycles, and more generally, $r_k$ stands for $r$ $k$-cycles.
Let $[G^{\alpha\beta}]$ be the inverse of the matrix $[G_{\alpha\beta}]$.

Given a function $\mathcal{F}$ with inputs $\tau=(\tau_1,\tau_2,\cdots) \in K^\infty$, define $R_r$ by
\[(R_r\mathcal{F})(\tau_1, \cdots, \tau_k,\cdots) = \mathcal{F}(\tau_r, \cdots, \tau_{kr},\cdots).\]

We define
\[S(q)\phi = \sum_{\alpha,\beta} \left((\phi,\phi_\alpha)+\left<\!\left< \frac{\phi}{1-L/q}, \phi_\alpha\right>\!\right>_{0,2}\right)G^{\alpha\beta}\phi_\beta.\]
Then $S(q)$ (originally defined on $K$) extends $\mathbb{Q}((q))$-linearly to $\mathcal{K}= K \otimes \mathbb{Q}((q))$, which we also denote by $S(q)$. 
We define the $\S$-operator as the block diagonal operator, with diagonal blocks $\S_r = R_r(S)$.

Let $v = \bar{v}$ be the vector with all components $1-q$.
Define $\bar{\t}$ such that $\bar{v}+\bar{\t} = [\S(v+\t)]_+$, the Laurent polynomial part of $\S(v+\t)$.
Let $F_1(\tau) = \mathcal{F}_1(\tau)$, the no-descendant $g=1$ potential with input $\t=\tau$.

\begin{prop} (Ancestor-Descendant correspondence in $g=1$)
\[\mathcal{F}_1(\t) = F_1(\tau) + \mathcal{\bar{F}}_1(\bar{\t}).\]
\end{prop}

We briefly explain how we use ancestor-descendant correspondence.
In our computations, we pick $\tau = (\tau_1,\tau_2, \cdots)$ so that $\bar{\t}(1)=0$, which is computed recursively in \cite{Tang2} Section 3.3 (see also Section 1.7 in this paper).
If all $\t_r$'s contain the pullback of $L-1$ from Deligne-Mumford spaces $\M_{1,n}$ along the forgetful map $\M_{1,n+\bar{n}}(X,d) \to \M_{1,n}$, (quite often) the invariants $\mathcal{\bar{F}}_1(\bar{\t})$ could be factorized into two parts for dimension reasons:
\begin{enumerate}
\item
The $K$-theoretical intersections of $L-1$ over the fixed loci of various permutation group action on $\M_{1,n}$;
\item
The no-descendant invariants contributed from a fiber of the forgetful map.
\end{enumerate}
For more details, please refer to Section 2.1.

\subsection{Statement of the result}
\

Let $\bar{\x}_r=1-q+\bar{\t}_r$, $t_{r,0} = \t_r(1)$ and $t_{r,0}^1$ be the coefficient of the $1$ ($\in K^0(X)$) component.
Let $\bar{\t}^{new}_2(q)$ be $-\bar{\t}_2(q)/(1-q)$ with $q-1$ changed to $q+1$, and
\[\begin{array}{ll}
\y_r&= \bar{\x}_1(q)+\dsum_{\alpha,\beta}\phi_\alpha G_1^{\alpha\beta}\<\!\<\phi_\beta,\dfrac{\partial \tau_r}{\partial t_{r,0}^1}-1, \bar{\x}_1(q)\>\!\>_{0,1_1+1_r+1_1},\\
\y_2^L & = \bar{\x}_1(q)+\dsum_{\alpha,\beta}\phi_\alpha G_1^{\alpha\beta}\<\!\<\phi_\beta,\bar{\t}^{new}_2(q), \bar{\x}_1(q)\>\!\>_{0,1_1+1_2+1_1}.
\end{array}\]
Define
\[\begin{array}{ll}
\mathcal{F}_{1,2}^{perm} (x) = & \dfrac1{24} \<\!\<\dfrac{\y_2^L(x^{-1})}{1-x\bar{L}},\y_2(x^{-1}),\y_2(x^{-1}),\y_2(x^{-1})\>\!\>_{1,4_1}; \\
\mathcal{F}_{1,3}^{perm} (x)  = & \dfrac16 \<\!\<\dfrac{\y_3(x^{-1})}{1-x\bar{L}}, \bar{\x}_1(x^{-1}),\bar{\x}_1(x^{-1})\>\!\>_{1,3_1}; \\
\mathcal{F}_{1,4}^{perm} (x) = & \dfrac14 \<\!\<\dfrac{\y_4(x^{-1})}{1-x\bar{L}}, \bar{\x}_1(x^{-1}), \bar{\x}_2(x^{-2})\>\!\>_{1,2_1+1_2}; \\
\mathcal{F}_{1,6}^{perm} (x) = &\dfrac16 \<\!\<\dfrac{\y_6(x^{-1})}{1-x\bar{L}}, \bar{\x}_2(x^{-2}), \bar{\x}_3(x^{-3}) \>\!\>_{1,1_1+1_2+1_3}.
\end{array}\]

We remark that the $\ell$'s of the double-bracket correlators in $\mathcal{F}_{1, M}^{perm}$ record the number of ramification points (carrying $\x_r(x^{-1})$'s) of elliptic curves under the order $M$ permutation of marked points.

\begin{thm}
The genus $1$ generating function 
\[\mathcal{F}_1(\t) = F_1(\tau) + \frac1{24} \log\left(\frac{\partial \tau_1}{\partial t_{1,0}}\right) + \sum_{M=2,3,4,6} \sum_{a = 0,\infty}Res_a \mathcal{F}_{1,M}^{perm}(x) \frac{dx}x.\]
\end{thm}

\section*{Acknowledgments}
\

The author thanks Professor Alexander Givental for introducing him to the $g=1$ quantum K-theory, suggesting related problems, and providing patience and guidance throughout the process.

\section{Collection of known results}
\

We recollect the following facts from \cite{Tang1}\cite{Tang2}.

\subsection{Basic theory}
\

For permutation equivariant quantum K-invariants, we also have the String, Dilaton, and WDVV equations regarding inputs for cycles of length $1$ \cite{Lee01}\cite{perm7} (for our version, see \cite{Tang2} Section 1).

\begin{enumerate}
\item
(String equation)
Set 
\[D\t(q)=\frac{\t(q)-\t(1)}{q-1}.\]
Then
\[\begin{array}{ll}
&\left<1, \t_{1,1}(L), \cdots, \t_{1,\ell_1}(L); \cdots, \t_{r,k}(L), \cdots \right>_{0,1_1+\ell,d} \\
=& \left<\t_{1,1}(L), \cdots, \t_{1,\ell_1}(L); \cdots, \t_{r,k}(L), \cdots \right>_{0,\ell,d} \\
+& \dsum_{i=1}^{\ell_1} \left<\t_{1,1}(L), \cdots, D\t_{1,i}(L) \cdots, \t_{1,\ell_1}(L); \cdots, \t_{r,k}(L), \cdots \right>_{0,\ell,d}.
\end{array}\]
\item
(Dilaton equation)
\[\begin{array}{ll}
&\left<L-1, \t_{1,1}(L), \cdots, \t_{1,\ell_1}(L); \cdots, \t_{k,i}(L), \cdots \right>_{0,1_1+\ell,d} \\
=& (\ell_1-2)\left<\t_{1,1}(L), \cdots, \t_{1,\ell_1}(L); \cdots, \t_{k,i}(L), \cdots \right>_{0,\ell,d}.
\end{array}\]
\item
(WDVV equation)
For all $\varphi,\psi \in K$,
\[\begin{array}{ll}
&(\varphi,\psi)+(1-xy)\left<\!\left< \dfrac{\varphi}{1-xL}, \dfrac{\psi}{1-yL}\right>\!\right>_{0,2_1} \\
=& \dsum_{\alpha,\beta} \left((\varphi,\phi_\alpha)+\left<\!\left< \dfrac{\varphi}{1-xL}, \phi_\alpha\right>\!\right>_{0,2_1}\right)G^{\alpha\beta}\left((\phi_\beta,\psi)+\left<\!\left< \phi_\beta, \dfrac{\psi}{1-yL}\right>\!\right>_{0,2_1}\right).
\end{array}\]
\end{enumerate}

\subsection{The permutation equivariant splitting axiom}
\

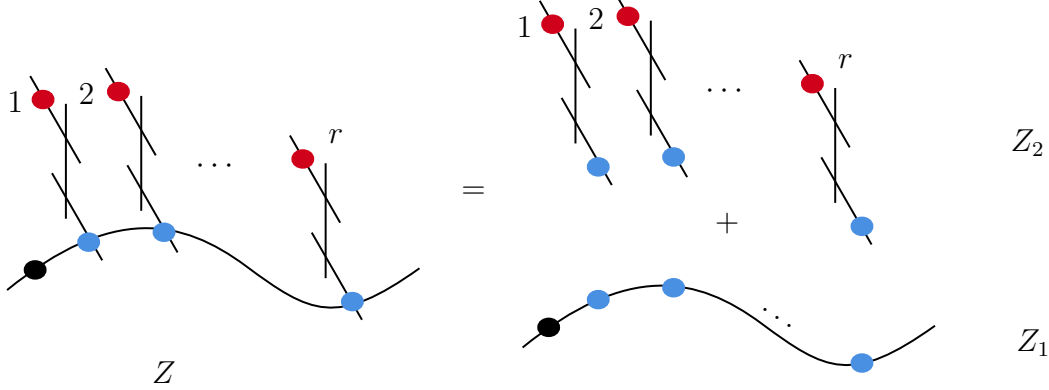
\tikzset{every picture/.style={line width=0.75pt}}
\begin{tikzpicture}[x=0.75pt,y=0.75pt,yscale=-1,xscale=1]
\draw    (14.33,154) .. controls (67.33,111) and (102.39,117.83) .. (134.86,141.92) .. controls (167.33,166) and (180.33,173) .. (222.33,143) ;
\draw  [color={rgb, 255:red, 0; green, 0; blue, 0 }  ,draw opacity=1 ][line width=3.75]  (25.24,143.83) .. controls (25.24,142.5) and (26.7,141.42) .. (28.49,141.42) .. controls (30.29,141.42) and (31.74,142.5) .. (31.74,143.83) .. controls (31.74,145.16) and (30.29,146.23) .. (28.49,146.23) .. controls (26.7,146.23) and (25.24,145.16) .. (25.24,143.83) -- cycle ;
\draw    (64,42) -- (89.33,86) ;
\draw    (75,91) -- (100.33,135) ;
\draw    (82,56) -- (82,114) ;
\draw  [color={rgb, 255:red, 208; green, 2; blue, 27 }  ,draw opacity=1 ][line width=3.75]  (67.24,53.83) .. controls (67.24,52.5) and (68.7,51.42) .. (70.49,51.42) .. controls (72.29,51.42) and (73.74,52.5) .. (73.74,53.83) .. controls (73.74,55.16) and (72.29,56.23) .. (70.49,56.23) .. controls (68.7,56.23) and (67.24,55.16) .. (67.24,53.83) -- cycle ;
\draw    (26,46) -- (51.33,90) ;
\draw    (37,95) -- (62.33,139) ;
\draw    (44,60) -- (44,118) ;
\draw  [color={rgb, 255:red, 208; green, 2; blue, 27 }  ,draw opacity=1 ][line width=3.75]  (29.24,57.83) .. controls (29.24,56.5) and (30.7,55.42) .. (32.49,55.42) .. controls (34.29,55.42) and (35.74,56.5) .. (35.74,57.83) .. controls (35.74,59.16) and (34.29,60.23) .. (32.49,60.23) .. controls (30.7,60.23) and (29.24,59.16) .. (29.24,57.83) -- cycle ;
\draw    (157,76) -- (182.33,120) ;
\draw    (168,125) -- (193.33,169) ;
\draw    (175,90) -- (175,148) ;
\draw  [color={rgb, 255:red, 208; green, 2; blue, 27 }  ,draw opacity=1 ][line width=3.75]  (160.24,87.83) .. controls (160.24,86.5) and (161.7,85.42) .. (163.49,85.42) .. controls (165.29,85.42) and (166.74,86.5) .. (166.74,87.83) .. controls (166.74,89.16) and (165.29,90.23) .. (163.49,90.23) .. controls (161.7,90.23) and (160.24,89.16) .. (160.24,87.83) -- cycle ;
\draw  [color={rgb, 255:red, 74; green, 144; blue, 226 }  ,draw opacity=1 ][line width=3.75]  (52.24,129.83) .. controls (52.24,128.5) and (53.7,127.42) .. (55.49,127.42) .. controls (57.29,127.42) and (58.74,128.5) .. (58.74,129.83) .. controls (58.74,131.16) and (57.29,132.23) .. (55.49,132.23) .. controls (53.7,132.23) and (52.24,131.16) .. (52.24,129.83) -- cycle ;
\draw  [color={rgb, 255:red, 74; green, 144; blue, 226 }  ,draw opacity=1 ][line width=3.75]  (90.24,124.83) .. controls (90.24,123.5) and (91.7,122.42) .. (93.49,122.42) .. controls (95.29,122.42) and (96.74,123.5) .. (96.74,124.83) .. controls (96.74,126.16) and (95.29,127.23) .. (93.49,127.23) .. controls (91.7,127.23) and (90.24,126.16) .. (90.24,124.83) -- cycle ;
\draw  [color={rgb, 255:red, 74; green, 144; blue, 226 }  ,draw opacity=1 ][line width=3.75]  (185.24,159.83) .. controls (185.24,158.5) and (186.7,157.42) .. (188.49,157.42) .. controls (190.29,157.42) and (191.74,158.5) .. (191.74,159.83) .. controls (191.74,161.16) and (190.29,162.23) .. (188.49,162.23) .. controls (186.7,162.23) and (185.24,161.16) .. (185.24,159.83) -- cycle ;
\draw    (274.33,183) .. controls (327.33,140) and (362.39,146.83) .. (394.86,170.92) .. controls (427.33,195) and (440.33,202) .. (482.33,172) ;
\draw  [color={rgb, 255:red, 0; green, 0; blue, 0 }  ,draw opacity=1 ][line width=3.75]  (284.24,172.83) .. controls (284.24,171.5) and (285.7,170.42) .. (287.49,170.42) .. controls (289.29,170.42) and (290.74,171.5) .. (290.74,172.83) .. controls (290.74,174.16) and (289.29,175.23) .. (287.49,175.23) .. controls (285.7,175.23) and (284.24,174.16) .. (284.24,172.83) -- cycle ;
\draw    (321,4) -- (346.33,48) ;
\draw    (332,53) -- (357.33,97) ;
\draw    (339,18) -- (339,76) ;
\draw  [color={rgb, 255:red, 208; green, 2; blue, 27 }  ,draw opacity=1 ][line width=3.75]  (324.24,15.83) .. controls (324.24,14.5) and (325.7,13.42) .. (327.49,13.42) .. controls (329.29,13.42) and (330.74,14.5) .. (330.74,15.83) .. controls (330.74,17.16) and (329.29,18.23) .. (327.49,18.23) .. controls (325.7,18.23) and (324.24,17.16) .. (324.24,15.83) -- cycle ;
\draw    (283,8) -- (308.33,52) ;
\draw    (294,57) -- (319.33,101) ;
\draw    (301,22) -- (301,80) ;
\draw  [color={rgb, 255:red, 208; green, 2; blue, 27 }  ,draw opacity=1 ][line width=3.75]  (286.24,19.83) .. controls (286.24,18.5) and (287.7,17.42) .. (289.49,17.42) .. controls (291.29,17.42) and (292.74,18.5) .. (292.74,19.83) .. controls (292.74,21.16) and (291.29,22.23) .. (289.49,22.23) .. controls (287.7,22.23) and (286.24,21.16) .. (286.24,19.83) -- cycle ;
\draw    (414,38) -- (439.33,82) ;
\draw    (425,87) -- (450.33,131) ;
\draw    (432,52) -- (432,110) ;
\draw  [color={rgb, 255:red, 208; green, 2; blue, 27 }  ,draw opacity=1 ][line width=3.75]  (417.24,49.83) .. controls (417.24,48.5) and (418.7,47.42) .. (420.49,47.42) .. controls (422.29,47.42) and (423.74,48.5) .. (423.74,49.83) .. controls (423.74,51.16) and (422.29,52.23) .. (420.49,52.23) .. controls (418.7,52.23) and (417.24,51.16) .. (417.24,49.83) -- cycle ;
\draw  [color={rgb, 255:red, 74; green, 144; blue, 226 }  ,draw opacity=1 ][line width=3.75]  (309.24,91.83) .. controls (309.24,90.5) and (310.7,89.42) .. (312.49,89.42) .. controls (314.29,89.42) and (315.74,90.5) .. (315.74,91.83) .. controls (315.74,93.16) and (314.29,94.23) .. (312.49,94.23) .. controls (310.7,94.23) and (309.24,93.16) .. (309.24,91.83) -- cycle ;
\draw  [color={rgb, 255:red, 74; green, 144; blue, 226 }  ,draw opacity=1 ][line width=3.75]  (347.24,86.83) .. controls (347.24,85.5) and (348.7,84.42) .. (350.49,84.42) .. controls (352.29,84.42) and (353.74,85.5) .. (353.74,86.83) .. controls (353.74,88.16) and (352.29,89.23) .. (350.49,89.23) .. controls (348.7,89.23) and (347.24,88.16) .. (347.24,86.83) -- cycle ;
\draw  [color={rgb, 255:red, 74; green, 144; blue, 226 }  ,draw opacity=1 ][line width=3.75]  (442.24,121.83) .. controls (442.24,120.5) and (443.7,119.42) .. (445.49,119.42) .. controls (447.29,119.42) and (448.74,120.5) .. (448.74,121.83) .. controls (448.74,123.16) and (447.29,124.23) .. (445.49,124.23) .. controls (443.7,124.23) and (442.24,123.16) .. (442.24,121.83) -- cycle ;
\draw  [color={rgb, 255:red, 74; green, 144; blue, 226 }  ,draw opacity=1 ][line width=3.75]  (309.24,158.83) .. controls (309.24,157.5) and (310.7,156.42) .. (312.49,156.42) .. controls (314.29,156.42) and (315.74,157.5) .. (315.74,158.83) .. controls (315.74,160.16) and (314.29,161.23) .. (312.49,161.23) .. controls (310.7,161.23) and (309.24,160.16) .. (309.24,158.83) -- cycle ;
\draw  [color={rgb, 255:red, 74; green, 144; blue, 226 }  ,draw opacity=1 ][line width=3.75]  (347.24,152.83) .. controls (347.24,151.5) and (348.7,150.42) .. (350.49,150.42) .. controls (352.29,150.42) and (353.74,151.5) .. (353.74,152.83) .. controls (353.74,154.16) and (352.29,155.23) .. (350.49,155.23) .. controls (348.7,155.23) and (347.24,154.16) .. (347.24,152.83) -- cycle ;
\draw  [color={rgb, 255:red, 74; green, 144; blue, 226 }  ,draw opacity=1 ][line width=3.75]  (442.24,190.83) .. controls (442.24,189.5) and (443.7,188.42) .. (445.49,188.42) .. controls (447.29,188.42) and (448.74,189.5) .. (448.74,190.83) .. controls (448.74,192.16) and (447.29,193.23) .. (445.49,193.23) .. controls (443.7,193.23) and (442.24,192.16) .. (442.24,190.83) -- cycle ;

\draw (13,52.4) node [anchor=north west][inner sep=0.75pt]    {$1$};
\draw (49,48.4) node [anchor=north west][inner sep=0.75pt]    {$2$};
\draw (175,71.4) node [anchor=north west][inner sep=0.75pt]    {$r$};
\draw (108,86.4) node [anchor=north west][inner sep=0.75pt]    {$\cdots $};
\draw (242,100.4) node [anchor=north west][inner sep=0.75pt]    {$=$};
\draw (270,14.4) node [anchor=north west][inner sep=0.75pt]    {$1$};
\draw (306,10.4) node [anchor=north west][inner sep=0.75pt]    {$2$};
\draw (432,33.4) node [anchor=north west][inner sep=0.75pt]    {$r$};
\draw (365,48.4) node [anchor=north west][inner sep=0.75pt]    {$\cdots $};
\draw (370,113.4) node [anchor=north west][inner sep=0.75pt]    {$+$};
\draw (392,152.4) node [anchor=north west][inner sep=0.75pt]    {$\ddots $};
\draw (86,188.4) node [anchor=north west][inner sep=0.75pt]    {$Z$};
\draw (519,71.4) node [anchor=north west][inner sep=0.75pt]    {$Z_{2}$};
\draw (522,172.4) node [anchor=north west][inner sep=0.75pt]    {$Z_{1}$};

\end{tikzpicture}

Colors when printed in black and white:

\tikzset{every picture/.style={line width=0.75pt}}
\begin{tikzpicture}[x=0.75pt,y=0.75pt,yscale=-1,xscale=1]
\draw  [line width=3.75]  (23,21.67) .. controls (23,20.19) and (24.19,19) .. (25.67,19) .. controls (27.14,19) and (28.33,20.19) .. (28.33,21.67) .. controls (28.33,23.14) and (27.14,24.33) .. (25.67,24.33) .. controls (24.19,24.33) and (23,23.14) .. (23,21.67) -- cycle ;
\draw  [color={rgb, 255:red, 208; green, 2; blue, 27 }  ,draw opacity=1 ][fill={rgb, 255:red, 255; green, 255; blue, 255 }  ,fill opacity=1 ][line width=3.75]  (113,21.33) .. controls (113,19.86) and (114.19,18.67) .. (115.67,18.67) .. controls (117.14,18.67) and (118.33,19.86) .. (118.33,21.33) .. controls (118.33,22.81) and (117.14,24) .. (115.67,24) .. controls (114.19,24) and (113,22.81) .. (113,21.33) -- cycle ;
\draw  [color={rgb, 255:red, 74; green, 144; blue, 226 }  ,draw opacity=1 ][line width=3.75]  (193,20.33) .. controls (193,18.86) and (194.19,17.67) .. (195.67,17.67) .. controls (197.14,17.67) and (198.33,18.86) .. (198.33,20.33) .. controls (198.33,21.81) and (197.14,23) .. (195.67,23) .. controls (194.19,23) and (193,21.81) .. (193,20.33) -- cycle ;
\draw (43,12.33) node [anchor=north west][inner sep=0.75pt]   [align=left] {black};
\draw (133,12) node [anchor=north west][inner sep=0.75pt]   [align=left] {red};
\draw (209,13) node [anchor=north west][inner sep=0.75pt]   [align=left] {blue};
\end{tikzpicture}

We decompose $Z=Z_1 \times_{\Delta^r} Z_2$ as in the figure above.
Let $h_{1,r}, h_{2,r}$ be the cyclic permutations of the blue marked points, defined according to the permutation of components containing the red points.
Let $h_1,h_2$ be the action of $h$ on the other marked points in $Z_1$ and $Z_2$.
Then the induced $h_1h_2$-action on $Z=Z_1 \times_{\Delta^r}Z_2$ is $h$.
Let $V_1,V_2$ be (virtual) $h_1,h_2$-bundles over $Z_1,Z_2$ and $V = V_1 \boxtimes V_2$.
Use $ev_{1,i}$ and $ev_{2,i}$ for the evaluation maps at marked points permuted by $h_{1,r}$ and $h_{2,r}$, respectively.

Given a function $\mathcal{F}$ with inputs $\tau=(\tau_1,\tau_2,\cdots) \in K^\infty$, recall that we defined $R_r$ by shifting $\tau$ inputs:
\[(R_r\mathcal{F})(\tau_1, \cdots, \tau_k,\cdots) = \mathcal{F}(\tau_r, \cdots, \tau_{kr},\cdots).\]

We define $G_r=R_rG$.
Then
\[str_{h} H^\ast (Z,V) = \sum_{\alpha,\beta} G_r^{\alpha\beta} str_{h_1h_{1,r}} H^\ast (Z_1,V_1 \cdot \prod_i ev_{1,i}^\ast \phi_\alpha) \cdot str_{h_2h_{2,r}} H^\ast (Z_2,V_2 \cdot \prod_i ev_{2,i}^\ast \phi_\beta).\]

\subsection{A formula for super-traces}
\

We recall the following proposition from \cite{perm9}.

Let $h$ be a generator of a cyclic group $H$, $\mathcal{M}$ an $H$-orbifold, and $V \to \mathcal{M}$ an orbi-bundle.
The super-trace $str_h H^\ast(\mathcal{M}, V)$ can be expressed in the following two ways.

\begin{enumerate}
\item
Set $E_h = \sum_\lambda \lambda^{-1} \mathbb{C}_{\lambda}$, where $\mathbb{C}_\lambda$ is the irreducible representation of $H$ on which $h$ acts as $\lambda$.
Then $str_h H^\ast(\mathcal{M}, V)$ equals the holomorphic Euler characteristic
\[\chi\left(\mathcal{M}/H, (V \cdot E_h)/H\right) = \bar{p}_\ast \left((V \cdot E_h)/H\right),\]
where $\bar{p}$ is the map from $\mathcal{M}/H$ to a point.
\item
The Kawasaki-Riemann-Roch type integral
\[\chi^{fake}\left(I^h\mathcal{M}, str_{\tilde{h}} \frac{V}{\wedge^\ast \tilde{N}^\ast_{I^h\mathcal{M}/\mathcal{M}}}\right).\]
Here, the fake Euler characteristic is defined by 
\[\chi^{fake}(M,V) = \int_{[M]} ch(V) \cdot td (M).\]

The (virtual) orbifold $I^h\mathcal{M}$ is interpreted as follows.
In $I(\mathcal{M}/H)$, each stratum is characterized as the lift of some element $h^k \in H$ with respect to a stratum in $I\mathcal{M}$ labeled by $g$.
We label this stratum by $\tilde{h} = (g,h^k)$.
Then $I^h\mathcal{M}$ is the sub-orbifold of $I\mathcal{M}$ corresponding to the strata in $I(\mathcal{M}/H)$ associated specifically with $h=h^1$.

Locally, one can describe $I^h\mathcal{M}$ as follows.
Take a point $x \in \mathcal{M}^h$ and choose a local chart $U/G(x)$ centered at $x$.
The action of $h$ lifts in $|G(x)|$-ways to automorphisms of $U$.
Each such lift fixes some points in $U$.
Then, near $x$, $I^h\mathcal{M}$ is obtained by taking the union of the these fixed loci and quotienting by the action of $H$.
\end{enumerate}

We refer to each term in the fake Euler characteristics as the {\em fake super-trace}:
\[str_{\tilde{h}}^{fake} H^\ast(\mathcal{M},V) = \chi^{fake}\left(\mathcal{M}^h, str_{\tilde{h}} \frac{V}{\wedge^\ast \tilde{N}^\ast_{\mathcal{M}^h/\mathcal{M}}}\right).\]

\subsection{A vanishing theorem and non-vanishing cases}
\

We say that an $h$-invariant nodal curve is {\em balanced} if it lies in the closure of smooth curves with $h$-action.
\cite{Tang2} Lemmas 3.1 and 4.1 are as follows.

\begin{enumerate}
\item
Assume that all $\bar{\t}_{r,k}$ vanish at $1$, except possibly for $\bar{\t}_{1,1},\bar{\t}_{1,2}$.
Then the correlator 
\[\left<\!\left<\bar{\mathbf{t}}_{1,1}(\bar{L}), \cdots, \bar{\mathbf{t}}_{1,\ell_1}(\bar{L}); \cdots, \bar{\mathbf{t}}_{r,k}(\bar{L}), \cdots\right>\!\right>_{0,\ell}\]
vanishes.
\item
The pre-image of unbalanced curves under the forgetful map $ft: \M_{1,n+\bar{n}}(X,d) \to \M_{1,n}$ does not contribute to $\bar{\mathcal{F}}_1(\bar{\t})$.
\end{enumerate}

All strata that contributes non-trivially to $\bar{\mathcal{F}}_1(\bar{\t})$ are listed in the following table:

\begin{table}[H]
\centering       
\begin{tabular}{c | c | c | c | c}                       
\# & $ord(h)$ & $\M$ & Balanced curves in $\M^h$ & Components in $ft^{-1}(\M_0^M)$ \\ 
\hline\hline
&&&&\\
$1$ & $1$ & $\M_{1,(\ell_1)_1} (X,d)$ & $\M_{\ell_1}^1(X,d)$ & $\M^1_{1,(\ell_1)_1} (X,d)$ \\
$2$ & $2$ & $\M_{1,4_1+(\ell_2)_2} (X,d)$ & $\M_{\ell_2}^2 (X,d)$ & $\M^{-1}_{1,4_1+(\ell_2)_2} (X,d)$ \\
$3$ & $4$ & $\M_{1,2_1+1_2+(\ell_4)_4} (X,d)$ & $\M_{\ell_4}^4 (X,d)$ & $\M_{1,2_1+1_2+(\ell_4)_4}^{\pm i} (X,d)$ \\
$4$ & $6$ & $\M_{1,1_1+1_2+1_3+(\ell_6)_6} (X,d)$ & $\M_{\ell_6}^6 (X,d)$ & $\M_{1,1_1+1_2+1_3+(\ell_6)_6}^{\omega^\pm} (X,d)$ \\
$5$ & $3$ & $\M_{1,3_1+(\ell_3)_3} (X,d)$ & $\M_{\ell_3}^3 (X,d)$ & $\M_{1,3_1+(\ell_3)_3}^{\omega^{\pm 2}} (X,d)$ \\ 
\end{tabular} 
\end{table} 

The superscripts in the last column indicate the eigenvalue of the symmetry acting on the cotangent spaces of level $1$ marked points on base curves.

We omit $(X,d)$ if $X=pt$ and $d=0$.

We expand
\[\begin{array}{llll}
\bar{\t}_r(q) & =\bar{t}_{r,1} &+ \bar{t}'_{r,1}(q-1) &+ O((q-1)^2);\\
\bar{\x}_r(q) & =\bar{x}_{r,\zeta} &+ \bar{x}'_{r,\zeta}(q/\zeta-1) &+ O((q/\zeta-1)^2).
\end{array}\]

All other super-traces vanish except for the following cases:
\begin{table}[htbp]
\centering       
\begin{tabular}{c | c | c}                       
Case & $ord(h)$ & Inputs \\
\hline\hline
&&\\
$1$ & $1$ & $\bar{t}_{1,1} \cdot (q-1)$  \\   
$2a$ & $2$ & In general $\bar{x}_{1,-1}, \bar{t}_{2,1}(q-1)$, one point $\bar{x}'_{1,-1}(-q-1)$ \\
$2c$ & $2$ & In general $\bar{x}_{1,-1}, \bar{t}_{2,1}(q-1)$, one pair $\bar{t}'_{2,1}(q-1)^2$ \\
$2b$ & $2$ & $\bar{x}_{1,-1}, \bar{t}_{2,1}(q-1)$\\
$3\pm$ & $4$ & $\bar{x}_{1,\pm i}, \bar{x}_{2,-1}, \bar{t}_{4,1} (q-1)$\\
$4\pm$ & $6$ & $\bar{x}_{1,\omega^{\pm 1}}, \bar{x}_{2,\omega^{\pm 2}}, \bar{x}_{3,-1}, \bar{t}_{6,1} (q-1)$\\
$5\pm$ & $3$ & $\bar{x}_{1,\omega^{\pm 2}}, \bar{t}_{6,1} (q-1)$\\
\end{tabular}
\end{table}

\subsection{Dilaton equations for permutable inputs}
\

By consecutively applying the dilaton equations (Proposition 4.1 in \cite{Tang2}), Section 4.3 in \cite{Tang2} gives the following results on super-traces:

\begin{enumerate}
\item
\[str_h H^\ast\left(\M_{1,4+2\ell_2}^{-1}, (-L_1-1) \cdot \prod_{i=1}^{\ell_2}(L_{2i+3}L_{2i+4}-1)\right) = \frac14 \cdot 2^{\ell_2} (\ell_2+1)!\]
\item
\[str_h H^\ast\left(\M_{1,4+2\ell_2}^{-1}, (L_5L_6-1)^2 \cdot \prod_{i=2}^{\ell_2}(L_{2i+3}L_{2i+4}-1)\right) = \frac14 \cdot 2^{\ell_2} (\ell_2+1)!\]
\item
\[str_h H^\ast\left(\M_{1,r+r\ell_r}^\zeta, \prod_{i=1}^{\ell_r}(L_{ri+1}\cdots L_{ri+r}-1)\right) = r^{\ell_r} \ell_r!\]
for $(r,\zeta)=(3,\omega^{\pm 2}), (4,i^{\pm}), (6,\omega^{\pm})$.
\end{enumerate}

\subsection{A proposition on contribution from chains of rational components}
\

We define the matrix 
\[A_r = \left[\sum_\gamma G_r^{\alpha\gamma} R_r \<\!\<\phi_\gamma, \bar{t}_1, \phi_\beta\>\!\>_{0,3_1}\right]\]
whose rows and columns are labeled by $\alpha,\beta$.
Identical to Proposition 1.1 in \cite{Tang1}, we have
\[(I-A_r)^{-1} = \left[\frac{\partial \tau_r^\alpha}{\partial t_{r,0}^\beta}\right].\]
In the right-hand side matrix, rows and columns are likewise labeled by $\alpha,\beta$.

Let `$\sim$' stand for curves (as base curves, after forgetting all marked points with input $\tau$'s) of the following form (or with $1$ red marked points, shrunk to the leftmost component; or with no red marked point, and the blue marked point shrunk directly to the node), connected to another node along the node with the arrow.

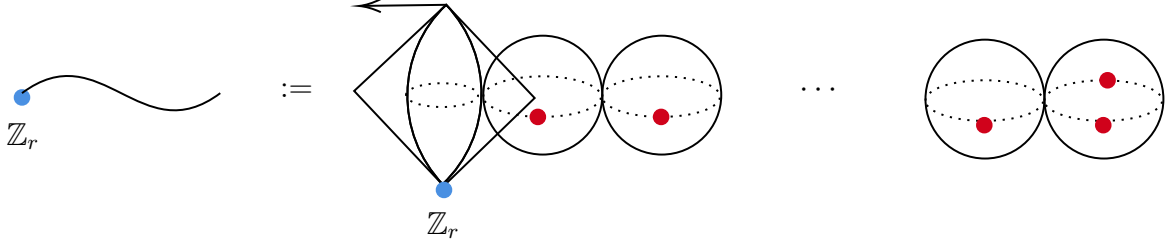
\begin{tikzpicture}[x=0.75pt,y=0.75pt,yscale=-1,xscale=1]
\draw  [draw opacity=0] (280.17,113.69) .. controls (280.02,113.54) and (279.87,113.38) .. (279.72,113.23) .. controls (255.09,87.24) and (255.79,46.58) .. (281.28,22.41) -- (325.89,69.46) -- cycle ; \draw   (280.17,113.69) .. controls (280.02,113.54) and (279.87,113.38) .. (279.72,113.23) .. controls (255.09,87.24) and (255.79,46.58) .. (281.28,22.41) ;  
\draw  [draw opacity=0] (281.28,22.41) .. controls (305.65,48.41) and (304.87,88.87) .. (279.46,112.95) .. controls (279.41,113) and (279.37,113.04) .. (279.32,113.08) -- (234.86,65.9) -- cycle ; \draw   (281.28,22.41) .. controls (305.65,48.41) and (304.87,88.87) .. (279.46,112.95) .. controls (279.41,113) and (279.37,113.04) .. (279.32,113.08) ;  
\draw  [dash pattern={on 0.84pt off 2.51pt}] (261,68) .. controls (261,64.69) and (269.13,62) .. (279.17,62) .. controls (289.2,62) and (297.33,64.69) .. (297.33,68) .. controls (297.33,71.31) and (289.2,74) .. (279.17,74) .. controls (269.13,74) and (261,71.31) .. (261,68) -- cycle ;
\draw   (300,68) .. controls (300,51.43) and (313.43,38) .. (330,38) .. controls (346.57,38) and (360,51.43) .. (360,68) .. controls (360,84.57) and (346.57,98) .. (330,98) .. controls (313.43,98) and (300,84.57) .. (300,68) -- cycle ;
\draw  [dash pattern={on 0.84pt off 2.51pt}] (300,68.75) .. controls (300,63.09) and (313.28,58.5) .. (329.67,58.5) .. controls (346.05,58.5) and (359.33,63.09) .. (359.33,68.75) .. controls (359.33,74.41) and (346.05,79) .. (329.67,79) .. controls (313.28,79) and (300,74.41) .. (300,68.75) -- cycle ;
\draw   (360,68) .. controls (360,51.43) and (373.43,38) .. (390,38) .. controls (406.57,38) and (420,51.43) .. (420,68) .. controls (420,84.57) and (406.57,98) .. (390,98) .. controls (373.43,98) and (360,84.57) .. (360,68) -- cycle ;
\draw  [dash pattern={on 0.84pt off 2.51pt}] (360,68.75) .. controls (360,63.09) and (373.28,58.5) .. (389.67,58.5) .. controls (406.05,58.5) and (419.33,63.09) .. (419.33,68.75) .. controls (419.33,74.41) and (406.05,79) .. (389.67,79) .. controls (373.28,79) and (360,74.41) .. (360,68.75) -- cycle ;
\draw   (523,70) .. controls (523,53.43) and (536.43,40) .. (553,40) .. controls (569.57,40) and (583,53.43) .. (583,70) .. controls (583,86.57) and (569.57,100) .. (553,100) .. controls (536.43,100) and (523,86.57) .. (523,70) -- cycle ;
\draw  [dash pattern={on 0.84pt off 2.51pt}] (523,70.75) .. controls (523,65.09) and (536.28,60.5) .. (552.67,60.5) .. controls (569.05,60.5) and (582.33,65.09) .. (582.33,70.75) .. controls (582.33,76.41) and (569.05,81) .. (552.67,81) .. controls (536.28,81) and (523,76.41) .. (523,70.75) -- cycle ;
\draw  [color={rgb, 255:red, 208; green, 2; blue, 27 }  ,draw opacity=1 ][line width=3]  (325.33,79) .. controls (325.33,77.8) and (326.3,76.83) .. (327.5,76.83) .. controls (328.7,76.83) and (329.67,77.8) .. (329.67,79) .. controls (329.67,80.2) and (328.7,81.17) .. (327.5,81.17) .. controls (326.3,81.17) and (325.33,80.2) .. (325.33,79) -- cycle ;
\draw  [color={rgb, 255:red, 208; green, 2; blue, 27 }  ,draw opacity=1 ][line width=3]  (387.5,79) .. controls (387.5,77.8) and (388.47,76.83) .. (389.67,76.83) .. controls (390.86,76.83) and (391.83,77.8) .. (391.83,79) .. controls (391.83,80.2) and (390.86,81.17) .. (389.67,81.17) .. controls (388.47,81.17) and (387.5,80.2) .. (387.5,79) -- cycle ;
\draw  [color={rgb, 255:red, 208; green, 2; blue, 27 }  ,draw opacity=1 ][line width=3]  (550.5,83.17) .. controls (550.5,81.97) and (551.47,81) .. (552.67,81) .. controls (553.86,81) and (554.83,81.97) .. (554.83,83.17) .. controls (554.83,84.36) and (553.86,85.33) .. (552.67,85.33) .. controls (551.47,85.33) and (550.5,84.36) .. (550.5,83.17) -- cycle ;
\draw   (583,70) .. controls (583,53.43) and (596.43,40) .. (613,40) .. controls (629.57,40) and (643,53.43) .. (643,70) .. controls (643,86.57) and (629.57,100) .. (613,100) .. controls (596.43,100) and (583,86.57) .. (583,70) -- cycle ;
\draw  [dash pattern={on 0.84pt off 2.51pt}] (583,70.75) .. controls (583,65.09) and (596.28,60.5) .. (612.67,60.5) .. controls (629.05,60.5) and (642.33,65.09) .. (642.33,70.75) .. controls (642.33,76.41) and (629.05,81) .. (612.67,81) .. controls (596.28,81) and (583,76.41) .. (583,70.75) -- cycle ;
\draw  [color={rgb, 255:red, 208; green, 2; blue, 27 }  ,draw opacity=1 ][line width=3]  (610.5,83.17) .. controls (610.5,81.97) and (611.47,81) .. (612.67,81) .. controls (613.86,81) and (614.83,81.97) .. (614.83,83.17) .. controls (614.83,84.36) and (613.86,85.33) .. (612.67,85.33) .. controls (611.47,85.33) and (610.5,84.36) .. (610.5,83.17) -- cycle ;
\draw  [color={rgb, 255:red, 208; green, 2; blue, 27 }  ,draw opacity=1 ][line width=3]  (612.67,60.5) .. controls (612.67,59.3) and (613.64,58.33) .. (614.83,58.33) .. controls (616.03,58.33) and (617,59.3) .. (617,60.5) .. controls (617,61.7) and (616.03,62.67) .. (614.83,62.67) .. controls (613.64,62.67) and (612.67,61.7) .. (612.67,60.5) -- cycle ;
\draw  [color={rgb, 255:red, 74; green, 144; blue, 226 }  ,draw opacity=1 ][line width=3]  (278,115.86) .. controls (278,114.66) and (278.97,113.69) .. (280.17,113.69) .. controls (281.36,113.69) and (282.33,114.66) .. (282.33,115.86) .. controls (282.33,117.05) and (281.36,118.02) .. (280.17,118.02) .. controls (278.97,118.02) and (278,117.05) .. (278,115.86) -- cycle ;
\draw    (281.28,22.41) -- (239.33,22.97) ;
\draw [shift={(237.33,23)}, rotate = 359.23] [color={rgb, 255:red, 0; green, 0; blue, 0 }  ][line width=0.75]    (10.93,-3.29) .. controls (6.95,-1.4) and (3.31,-0.3) .. (0,0) .. controls (3.31,0.3) and (6.95,1.4) .. (10.93,3.29)   ;
\draw  [color={rgb, 255:red, 74; green, 144; blue, 226 }  ,draw opacity=1 ][line width=3]  (65.17,69.17) .. controls (65.17,67.97) and (66.14,67) .. (67.33,67) .. controls (68.53,67) and (69.5,67.97) .. (69.5,69.17) .. controls (69.5,70.36) and (68.53,71.33) .. (67.33,71.33) .. controls (66.14,71.33) and (65.17,70.36) .. (65.17,69.17) -- cycle ;
\draw    (67.33,67) .. controls (107.33,37) and (127.33,97) .. (167.33,67) ;

\draw (458,60.4) node [anchor=north west][inner sep=0.75pt]    {$\cdots $};
\draw (270,125.4) node [anchor=north west][inner sep=0.75pt]    {$\mathbb{Z}_{r}$};
\draw (58,80.4) node [anchor=north west][inner sep=0.75pt]    {$\mathbb{Z}_{r}$};
\draw (196,60.4) node [anchor=north west][inner sep=0.75pt]    {$:=$};
\end{tikzpicture}

Let $\ell_r$ be the total number of red marked points.
We define the contribution of $\sim$ as the weighted (by $1/\ell_r!$) sum over all $\ell_r$ and possible labeling of marked points, of inputs that account for connecting the right-hand side curve (by the splitting axiom), with the order-$r$ symmetry acting as $\zeta$ on the blue $\mathbb{Z}_r$ ramification point.

\begin{prop}
The total contribution of $\sim$ is 
\[\bar{y}_{1,\zeta} = \bar{x}_{1,\zeta} + \sum_{\beta,\gamma} \phi_\beta G_1^{\beta\gamma} \<\!\<\phi_\gamma, \frac{\partial \tau_r}{\partial t_{r,0}^1} -1, \bar{x}_{1,\zeta}\>\!\>_{0,1_1+1_r+1_1}.\]
\end{prop}

\begin{proof}
Using
\[\sum_{\ell_r \geq 0} A_r^{\ell_r} = \left[\frac{\partial \tau^\alpha_r}{\partial t^\beta_{r,0}}\right]\]
and
\[R_r\<\!\< \phi_\alpha, \bar{t}_{r,1},\bar{t}_{r,1}\>\!\>_{0,3_1} = \sum_{\beta,\gamma} R_r \<\!\< \phi_\alpha, \bar{t}_{r,1},\phi_\beta\>\!\>_{0,3_1} G_r^{\beta\gamma} R_r \<\!\< \phi_\gamma, \bar{t}_{r,1},1\>\!\>_{0,3_r}\]
for $\ell_r \geq 2$, we see that $\bar{y}_{1,\zeta}$ expands as
\[\begin{array}{ll}
&\bar{x}_{1,\zeta} +  \dsum_{\beta,\gamma} \phi_\beta G_1^{\beta\gamma} \<\!\<\phi_\gamma, \bar{t}_{r,1}, \bar{x}_{1,\zeta}\>\!\>_{0,1_1+1_r+1_1}\\
+& \dsum_{\ell_r \geq 2}\left(\dsum_{\alpha,\beta} \phi_\alpha G_1^{\alpha\beta} \<\!\<\phi_\beta, \dsum_{\gamma,\delta,\epsilon}\phi_\gamma[A_r^{\ell_r-2}]_\delta^\gamma G_r^{\delta\epsilon} R_r\<\!\< \phi_\epsilon, \bar{t}_{r,1},\bar{t}_{r,1}\>\!\>_{0,3_1}, \bar{x}_{1,\zeta}\>\!\>_{0,1_1+1_r+1_1}\right)
\end{array}\]

The first two terms correspond to $\ell_r = 0,1$, while the summation precisely covers $\ell_r \geq 2$, by the splitting axiom.
\end{proof}

\subsection{An algorithm to compute $\tau$}
\

In \cite{Tang2} Section 3.3, we proved the following algorithm for $\tau$.
We define a map $T: K^\infty \to K^\infty$, by sending $^1\tau = (^1\tau_1,\cdots,^1\tau_r, \cdots)$ to $^2\tau = (^2\tau_1,\cdots,^2\tau_r, \cdots)$, where
\[^2\tau_r = \t_r(1) + R_r\left[\dsum_{\alpha,\beta} \<\!\<LD\t_r(L), \phi_\alpha\>\!\>_{0,2_1}G_r^{\alpha\beta}\phi_\beta\right].\]
In these double-bracket correlators, we take the input $\tau$ to be $^1\tau$, and 
\[D\t(L)=\frac{\t(1)-\t(L)}{1-L}.\]
Then $T$ is a contraction map: if $^1\tau - ^2\tau \in \Lambda_+^n$, then $T(^1\tau)-T(^2\tau) \in \Lambda_+^{n+1}$.
Here $\Lambda$ is a local algebra with maximal ideal $\Lambda_+$ containing $\t_r,\tau_r$ and Novikov variables.
This procedure recursively computes $\tau$ modulo any power of $\Lambda_+$ and shows that the sequence $\tau^{(n)} = T^n(0)$ converges to some $\tau \in K^\infty$ for which $\bar{\t}(1)=0$.

\section{Computations}

\subsection{Separating base curve and map contributions}
\

We present a formula that separates various super-traces into the product of $(a)$ the contributions of the base curves and $(b)$ the contribution of the map from a $fixed$ curve to $X$.

\begin{lemma}
Let $\pi: \mathcal{N} \to \mathcal{M}$ be an $h$-equivariant map between virtual $h$-orbifolds, and let $\tilde{h}$ be a lift of $h$.
Let $W \to \mathcal{M}, V \to \mathcal{N}$ be virtual $h$-bundles such that:
\begin{enumerate}
\item
$str_{\tilde{h}}^{fake} H^\ast (\pi^{-1}(\{C\}), V)$ is independent of the base point $\{C\} \in \mathcal{M}^h$.
\item
$ch (str_{\tilde{h}} W) \in H^{\dim_\mathbb{R} \mathcal{M}^h}(\mathcal{M}^h)$.
\end{enumerate}
Then
\[str^{fake}_{\tilde{h}} H^\ast (\mathcal{N}, \pi^\ast W \cdot V) = \frac{str^{fake}_{\tilde{h}} H^\ast (\mathcal{M}, W)}{str_{\tilde{h}}c_0(\wedge^\ast (N_{\mathcal{M}^h/\mathcal{M}}|_{\{C\}})^\ast)} \cdot str^{fake}_{\tilde{h}} H^\ast (\pi^{-1}(\{C\}), V)\]
for any $C \in \mathcal{M}^h$.
\end{lemma}

\begin{proof}
Recall that $E_h = \sum_\lambda \lambda \mathbb{C}_{\lambda^{-1}}$ satisfies $E^2_h = |H| E_h$, where $H$ is the group generated by $h$.
Observe that taking the quotient by the finite group $H$ does not affect normal bundles, and
\[N_{\mathcal{N}^h/\mathcal{N}} = N_{\text{fiber}^h/\text{fiber}} \oplus \pi^\ast N_{\mathcal{M}^h/\mathcal{M}}.\]
Hence, the left-hand side is
\[\begin{array}{ll}
&\dint_{\mathcal{N}^h/H} ch \ str_{\tilde{h}} \left(\dfrac{V \cdot \pi^\ast W \cdot E_h}{\wedge^\ast N_{\mathcal{N}^h/\mathcal{N}}^\ast}/H\right) td (\mathcal{N}^h/H)\\
=& \left(\dint_{\mathcal{M}^h/H} ch \ str_{\tilde{h}} \left(\dfrac{W \cdot E_h}{\wedge^\ast N_{\mathcal{M}^h/\mathcal{M}}^\ast}/H\right) td (\mathcal{M}^h/H)\right) \cdot str_{\tilde{h}}^{fake} H^\ast (\pi^{-1}(\{C\}), V)
\end{array}\]
by assumption (1) and any $C \in \mathcal{M}^h$.
By assumption (2) together with a dimension argument, only the degree-zero term $\frac1{str_{\tilde{h}}c_0(\wedge^\ast (N_{\mathcal{M}^h/\mathcal{M}}|_{\{C\}})^\ast)}$ of 
\[ch \ str_{\tilde{h}} \left(\dfrac{E_h}{\wedge^\ast N_{\mathcal{M}^h/\mathcal{M}}^\ast}/H\right) td (\mathcal{M}^h/H)\]
matters.
This completes the proof.
\end{proof}

We have the following direct corollary.
\begin{cor}
Further let $\mathcal{M}'$ be such that $\mathcal{M} \supseteq \mathcal{M}' \supseteq \mathcal{M}^h$.
If $str_{\tilde{h}} V$ is independent of different lifts $\tilde{h}$ of $h$ (i.e., it depends only on $h$), then
\[str^{fake}_{\tilde{h}} H^\ast (\mathcal{N}, \pi^\ast W \cdot V) = \frac1{str_{\tilde{h}}c_0(\wedge^\ast N_{\mathcal{M}'/\mathcal{M}}^\ast)} \cdot str^{fake}_{\tilde{h}} H^\ast (\mathcal{M}', W) str_{\tilde{h}} H^\ast (\pi^{-1}(\{C\}), V).\]
\end{cor}

\subsection{Non-permutative part: Case 1}
\

This part is identical to \cite{Tang1} Theorem $1$.
The total contribution is 
\[\frac1{24} \log \det\left(\frac{\partial \tau_1}{\partial t_{1,0}}\right).\]

\subsection{Case $2a$}
\

In Case $2a$, by Corollary 2.1 and Section 1.5, we have
\[\begin{array}{ll}
&\dfrac1{4!\ell_2!} \dsum_{\bar{\ell},d} \frac{Q^d}{\bar{\ell}!r^\ell} \cdot 4 \<\bar{x}'_{1,-1}(-\bar{L}-1), \bar{x}_{1,-1},\cdots; \bar{t}_{2,1} (\bar{L}-1), \cdots ; \tau_2, \cdots \>_{1,4_1+(\ell_2)_2+(\bar{\ell}_2)_2,d} \\
=&\dfrac1{c_0(\wedge^\ast N^\ast_{\M_{1,4+2\ell_2}^{-1}/\M_{1,4+2\ell_2}})} str_{h_2} H^\ast(\M^{-1}_{1,4+2\ell_2},W_{2a,\ell_2}) \cdot str_{h_2} H^\ast (\pi_2^{-1}(\{C_2\}), V_{2a,\ell_2}) \\
=&\dfrac16 \cdot \dfrac{(\ell_2+1)}{4 \cdot 8} \cdot str_{h_2} H^\ast (\pi_2^{-1}(\{C_2\}), V_{2a,\ell_2}),\end{array}\]
where
\[\begin{array}{ll}
W_{2a, \ell_2} & = \dfrac1{6 \cdot \ell_2! 2^{\ell_2}} (-L_1-1) \dprod_{i=1}^{\ell_2} (L_{2i+3}L_{2i+4} -1);\\
V_{2a,\ell_2} &= \dsum_{\bar{\ell},d} \frac{Q^d}{\prod_r \bar{\ell}_r!r^{\bar{\ell}_r}} \left(ev_1^\ast \bar{x}'_{1,-1} \prod_{i=2}^4 ev_i^\ast \bar{x}_{1,-1} \cdot \dprod_{i=1}^{\ell_2} (ev_{2i+3}^\ast \bar{t}_{2,1} ev_{2i+4}^\ast \bar{t}_{2,1}) \cdot (\bar{ft}_2)_\ast \mathcal{T}\right),
\end{array}\]
and $\mathcal{T}$ is the virtual bundle on $\M_{1,|\ell+\bar{\ell}|}(X,d)$ contributed from input $\tau$'s.

\begin{lemma}
\[\begin{array}{ll}
&\dfrac1{4 \cdot 8} \dsum_{\ell_2} (\ell_2+1)str_{h_2} H^\ast (\pi_2^{-1}(\{C_2\}), V_{2a,\ell_2}) \\
=& \<\!\<\bar{x}^{2a}_{1,-1}(-\bar{L}-1), \bar{y}_{1,-1}, \bar{y}_{1,-1}, \bar{y}_{1,-1}\>\!\>_{1,4_1},
\end{array}\]
where
\[\bar{x}^{2a}_{1,-1} = \bar{x}'_{1,-1} - \sum_{\alpha,\beta} \phi_\alpha G_1^{\alpha\beta}\<\!\<\phi_\beta,\bar{t}_{2,1},\bar{x}'_{1,-1}\>\!\>_{0,1_1+1_2+1_1}.\]
\end{lemma}

\begin{proof}
\

{\bf Step 1: Coefficients $\ell_2+1$.}
We compute via curves of the following type.
Here, the complex structure of $\square = T^2/\mathbb{Z}_2$ is fixed as $\Gamma$, as shown in the picture below.
This figure represents the total contribution (from the relevant moduli space) of all possible maps arising from curves with two $\sim$-chains.
We sum over every possible value of $\ell_2$ (the number of level-$2$ marked points) and all configurations of marked points, weighting each term by $\frac{1}{\ell_2!}$.

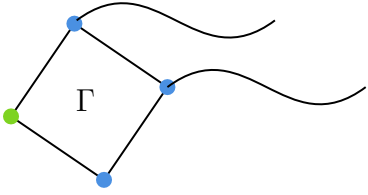
\begin{tikzpicture}[x=0.75pt,y=0.75pt,yscale=-1,xscale=1]
\draw   (23.74,70) -- (55.71,23.18) -- (102.61,55.21) -- (70.64,102.03) -- cycle ;
\draw  [color={rgb, 255:red, 74; green, 144; blue, 226 }  ,draw opacity=1 ][line width=3]  (54.58,24.83) .. controls (53.67,24.21) and (53.43,22.96) .. (54.06,22.05) .. controls (54.68,21.14) and (55.93,20.91) .. (56.84,21.53) .. controls (57.75,22.16) and (57.98,23.4) .. (57.36,24.31) .. controls (56.73,25.22) and (55.49,25.46) .. (54.58,24.83) -- cycle ;
\draw  [color={rgb, 255:red, 74; green, 144; blue, 226 }  ,draw opacity=1 ][line width=3]  (101.48,56.86) .. controls (100.57,56.24) and (100.33,54.99) .. (100.96,54.08) .. controls (101.58,53.17) and (102.83,52.94) .. (103.74,53.56) .. controls (104.65,54.19) and (104.88,55.43) .. (104.26,56.34) .. controls (103.63,57.25) and (102.39,57.49) .. (101.48,56.86) -- cycle ;
\draw  [color={rgb, 255:red, 126; green, 211; blue, 33 }  ,draw opacity=1 ][line width=3]  (22.6,71.65) .. controls (21.69,71.02) and (21.46,69.78) .. (22.09,68.87) .. controls (22.71,67.96) and (23.95,67.72) .. (24.87,68.35) .. controls (25.78,68.97) and (26.01,70.22) .. (25.39,71.13) .. controls (24.76,72.04) and (23.52,72.27) .. (22.6,71.65) -- cycle ;
\draw  [color={rgb, 255:red, 74; green, 144; blue, 226 }  ,draw opacity=1 ][line width=3]  (69.51,103.68) .. controls (68.59,103.05) and (68.36,101.81) .. (68.99,100.9) .. controls (69.61,99.99) and (70.86,99.75) .. (71.77,100.38) .. controls (72.68,101) and (72.91,102.25) .. (72.29,103.16) .. controls (71.66,104.07) and (70.42,104.3) .. (69.51,103.68) -- cycle ;
\draw    (56.84,21.53) .. controls (96.84,-8.47) and (116.84,51.53) .. (156.84,21.53) ;
\draw    (102.61,55.21) .. controls (142.61,25.21) and (162.61,85.21) .. (202.61,55.21) ;
\draw (55,55) node [anchor=north west][inner sep=0.75pt]    {$\Gamma $};
\end{tikzpicture}

All marked point configurations contribute equally for each fixed $\ell_2$.
The total number of configurations of marked points is
\[\sum_{\ell'_2} C_{\ell_2}^{\ell'_2} \ell'_2!(\ell_2-\ell'_2)!=(\ell_2+1)!,\]
which cancels the weight $1/\ell_2!$ for curves with $\ell_2$ level-2 marked points, leaving a coefficient $(\ell_2+1)$.
Hence, the contribution is the left-hand side of the lemma.

Alternatively, summing over all $\ell_2$'s, we can interpret the contribution as that of a $g=1$ curve with $\mathbb{Z}_2$ symmetry, together with $4$ marked points carrying inputs $\bar{x}'_{1,-1}(-\bar{L}-1)$ and $\bar{x}_{1,-1}$, plus the ``$\sim$'' nodes with input $\bar{y}_{1,-1}$, allowing arbitrarily many marked points at any level carrying inputs $\tau$'s.
Applying Corollary 2.1 and Section 1.5, we simplify the left-hand side of the statement as
\[\<\!\<\bar{x}'_{1,-1}(-\bar{L}-1), \bar{y}_{1,-1}, \bar{y}_{1,-1}, \bar{x}_{1,-1}\>\!\>_{1,4_1}.\]

{\bf Step 2: Modification.}

Notice that the contribution from curves

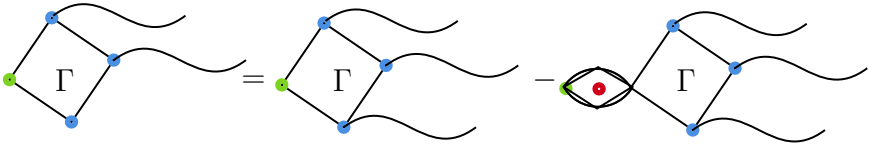
\begin{tikzpicture}[x=0.5pt,y=0.5pt,yscale=-1,xscale=1]
\draw   (7.74,70) -- (39.71,23.18) -- (86.61,55.21) -- (54.64,102.03) -- cycle ;
\draw  [color={rgb, 255:red, 74; green, 144; blue, 226 }  ,draw opacity=1 ][line width=3]  (38.58,24.83) .. controls (37.67,24.21) and (37.43,22.96) .. (38.06,22.05) .. controls (38.68,21.14) and (39.93,20.91) .. (40.84,21.53) .. controls (41.75,22.16) and (41.98,23.4) .. (41.36,24.31) .. controls (40.73,25.22) and (39.49,25.46) .. (38.58,24.83) -- cycle ;
\draw  [color={rgb, 255:red, 74; green, 144; blue, 226 }  ,draw opacity=1 ][line width=3]  (85.48,56.86) .. controls (84.57,56.24) and (84.33,54.99) .. (84.96,54.08) .. controls (85.58,53.17) and (86.83,52.94) .. (87.74,53.56) .. controls (88.65,54.19) and (88.88,55.43) .. (88.26,56.34) .. controls (87.63,57.25) and (86.39,57.49) .. (85.48,56.86) -- cycle ;
\draw  [color={rgb, 255:red, 126; green, 211; blue, 33 }  ,draw opacity=1 ][line width=3]  (6.6,71.65) .. controls (5.69,71.02) and (5.46,69.78) .. (6.09,68.87) .. controls (6.71,67.96) and (7.95,67.72) .. (8.87,68.35) .. controls (9.78,68.97) and (10.01,70.22) .. (9.39,71.13) .. controls (8.76,72.04) and (7.52,72.27) .. (6.6,71.65) -- cycle ;
\draw  [color={rgb, 255:red, 74; green, 144; blue, 226 }  ,draw opacity=1 ][line width=3]  (53.51,103.68) .. controls (52.59,103.05) and (52.36,101.81) .. (52.99,100.9) .. controls (53.61,99.99) and (54.86,99.75) .. (55.77,100.38) .. controls (56.68,101) and (56.91,102.25) .. (56.29,103.16) .. controls (55.66,104.07) and (54.42,104.3) .. (53.51,103.68) -- cycle ;
\draw    (40.84,21.53) .. controls (80.84,-8.47) and (100.84,51.53) .. (140.84,21.53) ;
\draw    (86.61,55.21) .. controls (126.61,25.21) and (146.61,85.21) .. (186.61,55.21) ;
\draw   (213.74,74) -- (245.71,27.18) -- (292.61,59.21) -- (260.64,106.03) -- cycle ;
\draw  [color={rgb, 255:red, 74; green, 144; blue, 226 }  ,draw opacity=1 ][line width=3]  (244.58,28.83) .. controls (243.67,28.21) and (243.43,26.96) .. (244.06,26.05) .. controls (244.68,25.14) and (245.93,24.91) .. (246.84,25.53) .. controls (247.75,26.16) and (247.98,27.4) .. (247.36,28.31) .. controls (246.73,29.22) and (245.49,29.46) .. (244.58,28.83) -- cycle ;
\draw  [color={rgb, 255:red, 74; green, 144; blue, 226 }  ,draw opacity=1 ][line width=3]  (291.48,60.86) .. controls (290.57,60.24) and (290.33,58.99) .. (290.96,58.08) .. controls (291.58,57.17) and (292.83,56.94) .. (293.74,57.56) .. controls (294.65,58.19) and (294.88,59.43) .. (294.26,60.34) .. controls (293.63,61.25) and (292.39,61.49) .. (291.48,60.86) -- cycle ;
\draw  [color={rgb, 255:red, 126; green, 211; blue, 33 }  ,draw opacity=1 ][line width=3]  (212.6,75.65) .. controls (211.69,75.02) and (211.46,73.78) .. (212.09,72.87) .. controls (212.71,71.96) and (213.95,71.72) .. (214.87,72.35) .. controls (215.78,72.97) and (216.01,74.22) .. (215.39,75.13) .. controls (214.76,76.04) and (213.52,76.27) .. (212.6,75.65) -- cycle ;
\draw  [color={rgb, 255:red, 74; green, 144; blue, 226 }  ,draw opacity=1 ][line width=3]  (259.51,107.68) .. controls (258.59,107.05) and (258.36,105.81) .. (258.99,104.9) .. controls (259.61,103.99) and (260.86,103.75) .. (261.77,104.38) .. controls (262.68,105) and (262.91,106.25) .. (262.29,107.16) .. controls (261.66,108.07) and (260.42,108.3) .. (259.51,107.68) -- cycle ;
\draw    (246.84,25.53) .. controls (286.84,-4.47) and (306.84,55.53) .. (346.84,25.53) ;
\draw    (292.61,59.21) .. controls (332.61,29.21) and (352.61,89.21) .. (392.61,59.21) ;
\draw    (260.64,106.03) .. controls (300.64,76.03) and (320.64,136.03) .. (360.64,106.03) ;
\draw   (477.74,76.12) -- (509.71,29.3) -- (556.61,61.33) -- (524.64,108.15) -- cycle ;
\draw  [color={rgb, 255:red, 74; green, 144; blue, 226 }  ,draw opacity=1 ][line width=3]  (508.58,30.95) .. controls (507.67,30.33) and (507.43,29.08) .. (508.06,28.17) .. controls (508.68,27.26) and (509.93,27.03) .. (510.84,27.65) .. controls (511.75,28.28) and (511.98,29.52) .. (511.36,30.43) .. controls (510.73,31.34) and (509.49,31.58) .. (508.58,30.95) -- cycle ;
\draw  [color={rgb, 255:red, 74; green, 144; blue, 226 }  ,draw opacity=1 ][line width=3]  (555.48,62.98) .. controls (554.57,62.36) and (554.33,61.11) .. (554.96,60.2) .. controls (555.58,59.29) and (556.83,59.06) .. (557.74,59.68) .. controls (558.65,60.31) and (558.88,61.55) .. (558.26,62.46) .. controls (557.63,63.37) and (556.39,63.61) .. (555.48,62.98) -- cycle ;
\draw  [color={rgb, 255:red, 126; green, 211; blue, 33 }  ,draw opacity=1 ][line width=3]  (427.54,78.31) .. controls (426.63,77.69) and (426.4,76.45) .. (427.02,75.53) .. controls (427.65,74.62) and (428.89,74.39) .. (429.8,75.01) .. controls (430.71,75.64) and (430.95,76.88) .. (430.32,77.8) .. controls (429.7,78.71) and (428.45,78.94) .. (427.54,78.31) -- cycle ;
\draw  [color={rgb, 255:red, 74; green, 144; blue, 226 }  ,draw opacity=1 ][line width=3]  (523.51,109.8) .. controls (522.59,109.17) and (522.36,107.93) .. (522.99,107.02) .. controls (523.61,106.11) and (524.86,105.87) .. (525.77,106.5) .. controls (526.68,107.12) and (526.91,108.37) .. (526.29,109.28) .. controls (525.66,110.19) and (524.42,110.42) .. (523.51,109.8) -- cycle ;
\draw    (510.84,27.65) .. controls (550.84,-2.35) and (570.84,57.65) .. (610.84,27.65) ;
\draw    (556.61,61.33) .. controls (596.61,31.33) and (616.61,91.33) .. (656.61,61.33) ;
\draw    (524.64,108.15) .. controls (564.64,78.15) and (584.64,138.15) .. (624.64,108.15) ;
\draw  [draw opacity=0] (427.02,75.53) .. controls (432.37,67.3) and (441.65,61.86) .. (452.19,61.86) .. controls (462.99,61.86) and (472.45,67.56) .. (477.74,76.12) -- (452.19,91.86) -- cycle ; \draw   (427.02,75.53) .. controls (432.37,67.3) and (441.65,61.86) .. (452.19,61.86) .. controls (462.99,61.86) and (472.45,67.56) .. (477.74,76.12) ;  
\draw  [draw opacity=0] (479.04,75.5) .. controls (473.86,84.5) and (464.15,90.55) .. (453.02,90.55) .. controls (441.91,90.55) and (432.21,84.51) .. (427.02,75.53) -- (453.02,60.55) -- cycle ; \draw   (479.04,75.5) .. controls (473.86,84.5) and (464.15,90.55) .. (453.02,90.55) .. controls (441.91,90.55) and (432.21,84.51) .. (427.02,75.53) ;  
\draw  [color={rgb, 255:red, 208; green, 2; blue, 27 }  ,draw opacity=1 ][line width=3]  (452.6,78.65) .. controls (451.69,78.02) and (451.46,76.78) .. (452.09,75.87) .. controls (452.71,74.96) and (453.95,74.72) .. (454.87,75.35) .. controls (455.78,75.97) and (456.01,77.22) .. (455.39,78.13) .. controls (454.76,79.04) and (453.52,79.27) .. (452.6,78.65) -- cycle ;
\draw (181,64.4) node [anchor=north west][inner sep=0.75pt]    {$=$};
\draw (401,60.4) node [anchor=north west][inner sep=0.75pt]    {$-$};
\draw (40,60) node [anchor=north west][inner sep=0.75pt]    {$\Gamma $};
\draw (250,60) node [anchor=north west][inner sep=0.75pt]    {$\Gamma $};
\draw (510,60) node [anchor=north west][inner sep=0.75pt]    {$\Gamma $};
\end{tikzpicture}

\noindent by counting the number of curves: $(\ell_2+1)! = (\ell_2+2)!-(\ell_2+1) \cdot (\ell_2+1)!$.
Directly computing the right-hand side's contribution yields the right-hand side of our claim.
\end{proof}

Thus, the total contribution is
\[\frac16 \<\!\<\bar{x}^{2a}_{1,-1}(-\bar{L}-1), \bar{y}_{1,-1}, \bar{y}_{1,-1}, \bar{y}_{1,-1}\>\!\>_{1,4_1}.\]

Alternatively, since the space of curves of each configuration is $0$-dimensional, switching (inputs for) the green marked point with one of the blue marked points does not affect the results.
This rewrites the contribution as 
\[\frac16 \<\!\<'\bar{x}^{2a}_{1,-1}(-\bar{L}-1), \bar{y}'_{1,-1}, \bar{y}_{1,-1}, \bar{y}_{1,-1}\>\!\>_{1,4_1},\]
where
\['\bar{x}^{2a}_{1,-1} = \bar{x}_{1,-1} - \sum_{\alpha,\beta} \phi_\alpha G_1^{\alpha\beta}\<\!\<\phi_\beta,\bar{t}_{2,1},\bar{x}_{1,-1}\>\!\>_{0,1_1+1_2+1_1}.\]

\subsection{Case $2b$}
\

In Case $2b$, we compute
\[\begin{array}{ll}
&\dsum_{\ell_2 \geq 0} \dfrac1{4!\ell_2!} \dsum_{\bar{\ell}_2,d} \frac{Q^d}{\bar{\ell}_2!} \cdot \ell_2 \<\bar{x}_{1,-1},\cdots; \bar{t}'_{2,1} (\bar{L}-1)^2, \bar{t}_{2,1} (\bar{L}-1), \cdots ; \tau_2, \cdots \>_{1,4_1+(\ell_2)_2+(\bar{\ell}_2)_2,d} \\
=& \dfrac1{c_0(\wedge^\ast N^\ast_{\M_{1,4+2\ell_2}^0/\M_{1,4+2\ell_2}})} \dsum_{\ell_2 \geq 0} str_{h_2} H^\ast(\M^{-1}_{1,4+2\ell_2},W_{2a,\ell_2}) \cdot str_{h_2} H^\ast (\pi_2^{-1}(\{C_2\}), V_{2a,\ell_2})\\
=&\dfrac1{24} \dfrac1{4 \cdot 8} \dsum_{\ell_2} \ell_2(\ell_2+1) \cdot str_{h_2} H^\ast (\pi_2^{-1}(\{C_2\}), V_{2a,\ell_2}),\end{array}\]
where
\[\begin{array}{ll}
W_{2a, \ell_2} & = \dfrac1{24 \cdot (\ell_2-1)! 2^{\ell_2}} (L_5L_6-1)^2 \dprod_{i=2}^{\ell_2} (L_{2i+3}L_{2i+4} -1);\\
V_{2a,\ell_2} &= \dsum_{\bar{\ell}_2,d} \frac{Q^d}{\bar{\ell}_2!2^{\bar{\ell}_2}} \left(\prod_{i=1}^4 ev_i^\ast \bar{x}_{1,-1} \cdot (ev_5^\ast \bar{t}'_{2,1} ev_6^\ast \bar{t}'_{2,1}) \cdot \dprod_{i=2}^{\ell_2} (ev_{2i+3}^\ast \bar{t}_{2,1} ev_{2i+4}^\ast \bar{t}_{2,1}) \cdot (\bar{ft}_2)_\ast \mathcal{T}\right).
\end{array}\]

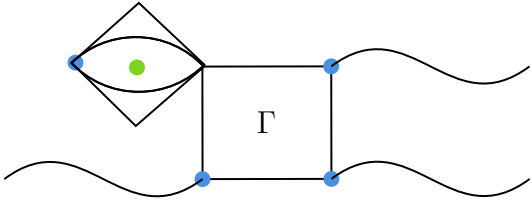
\begin{tikzpicture}[x=0.75pt,y=0.75pt,yscale=-1,xscale=1]
\draw   (223.97,28.83) -- (288.96,28.73) -- (289.05,85.52) -- (224.06,85.63) -- cycle ;
\draw  [color={rgb, 255:red, 74; green, 144; blue, 226 }  ,draw opacity=1 ][line width=3]  (222.06,85.63) .. controls (222.06,84.52) and (222.96,83.63) .. (224.06,83.63) .. controls (225.17,83.63) and (226.06,84.52) .. (226.06,85.63) .. controls (226.06,86.73) and (225.17,87.63) .. (224.06,87.63) .. controls (222.96,87.63) and (222.06,86.73) .. (222.06,85.63) -- cycle ;
\draw  [color={rgb, 255:red, 74; green, 144; blue, 226 }  ,draw opacity=1 ][line width=3]  (287.05,85.52) .. controls (287.05,84.42) and (287.95,83.52) .. (289.05,83.52) .. controls (290.16,83.52) and (291.05,84.42) .. (291.05,85.52) .. controls (291.05,86.63) and (290.16,87.52) .. (289.05,87.52) .. controls (287.95,87.52) and (287.05,86.63) .. (287.05,85.52) -- cycle ;
\draw  [color={rgb, 255:red, 74; green, 144; blue, 226 }  ,draw opacity=1 ][line width=3]  (157.97,26.83) .. controls (157.97,25.73) and (158.87,24.83) .. (159.97,24.83) .. controls (161.08,24.83) and (161.97,25.73) .. (161.97,26.83) .. controls (161.97,27.94) and (161.08,28.83) .. (159.97,28.83) .. controls (158.87,28.83) and (157.97,27.94) .. (157.97,26.83) -- cycle ;
\draw  [color={rgb, 255:red, 74; green, 144; blue, 226 }  ,draw opacity=1 ][line width=3]  (286.96,28.73) .. controls (286.96,27.63) and (287.86,26.73) .. (288.96,26.73) .. controls (290.07,26.73) and (290.96,27.63) .. (290.96,28.73) .. controls (290.96,29.83) and (290.07,30.73) .. (288.96,30.73) .. controls (287.86,30.73) and (286.96,29.83) .. (286.96,28.73) -- cycle ;
\draw  [color={rgb, 255:red, 126; green, 211; blue, 33 }  ,draw opacity=1 ][line width=3]  (189,29.26) .. controls (189,28.15) and (189.9,27.26) .. (191,27.26) .. controls (192.1,27.26) and (193,28.15) .. (193,29.26) .. controls (193,30.36) and (192.1,31.26) .. (191,31.26) .. controls (189.9,31.26) and (189,30.36) .. (189,29.26) -- cycle ;
\draw  [draw opacity=0] (225,27.88) .. controls (216.63,36.31) and (204.9,41.56) .. (191.91,41.56) .. controls (178.47,41.56) and (166.38,35.94) .. (157.96,26.99) -- (191.91,-3.29) -- cycle ; \draw   (225,27.88) .. controls (216.63,36.31) and (204.9,41.56) .. (191.91,41.56) .. controls (178.47,41.56) and (166.38,35.94) .. (157.96,26.99) ;  
\draw  [draw opacity=0] (157.96,26.99) .. controls (166.28,18.94) and (177.74,13.97) .. (190.39,13.97) .. controls (204.03,13.97) and (216.28,19.75) .. (224.71,28.93) -- (190.39,58.82) -- cycle ; \draw   (157.96,26.99) .. controls (166.28,18.94) and (177.74,13.97) .. (190.39,13.97) .. controls (204.03,13.97) and (216.28,19.75) .. (224.71,28.93) ;  
\draw    (288.96,28.73) .. controls (328.96,-1.27) and (348.96,58.73) .. (388.96,28.73) ;
\draw    (124.06,85.63) .. controls (164.06,55.63) and (184.06,115.63) .. (224.06,85.63) ;
\draw    (289.05,85.52) .. controls (329.05,55.52) and (349.05,115.52) .. (389.05,85.52) ;
\draw (250,50) node [anchor=north west][inner sep=0.75pt]    {$\Gamma $};
\end{tikzpicture}

Combining computations via this curve (with Corollary 2.1 and Section 1.5) and previous computations in Case $2a$, the total contribution in Case $2b$ becomes 
\[\frac1{24}\<\!\<\bar{x}^{2b}_{1,-1}(-\bar{L}-1), \bar{y}_{1,-1}, \bar{y}_{1,-1}, \bar{y}_{1,-1}\>\!\>_{1,4_1},\]
where
\[\bar{x}^{2b}_{1,-1} = \sum_{\alpha,\beta} \phi_\alpha G_1^{\alpha\beta}\<\!\<\phi_\beta,\bar{t}'_{2,1},\bar{x}_{1,-1}\>\!\>_{0,1_1+1_2+1_1}.\]

\subsection{Case 2c}
\

In this case, we compute the invariants
\[\dfrac1{4!\ell_2!} \dsum_{\bar{\ell},d} \frac{Q^d}{\bar{\ell}!} \<\bar{x}_{1,-1},\cdots; \bar{t}_{2,1} (\bar{L}-1), \cdots ; \tau, \cdots \>_{1,4_1+(\ell_2)_2+\bar{\ell},d}.\]
Our primary tool is Lemma 2.1.

\

{\bf Step 0: The map $\pi$ and $\mathcal{M,N}$.}

We describe pointed elliptic curves via equations $y^2z=x^3+axz^2+bz^3$ in $\mathbb{CP}^2$, with the marked point $[0:1:0]$ and the $\mathbb{Z}_2$ symmetry $[x:y:z] \to [x:-y:z]$.
Define a (generically double cover) map $\mathcal{M}_{\ell_2}^2 \to \mathcal{M}_{0,4+\ell_2}$ on the smooth locus by sending $[x:y:z]$ to $[x:z]=[y^2-bz^2:x^2+az^2]$.
This map induces $\M_{\ell_2}^2 \to \M_{0,4+\ell_2}$.
Composing with the forgetful map $\M_{0,4+\ell_2} \to \M_{0,3+\ell_2}$ that forgets the fourth marked point, we get 
\[\pi:\mathcal{N}=\M_{\ell_2}^2(X,d) \to \mathcal{M} = \M_{0,(3+\ell_2)_1}.\]

The first three level-$1$ marked points on $\M_{\ell_2}^2$ are mapped to the three marked points on $\M_{0,3+\ell_2}$, while the fourth is forgotten.
The $\ell_2$ pairs of marked points in $\M^2_{\ell_2}$ are mapped to the last $\ell_2$ marked points in $\M_{0,3+\ell_2}$.
The fiber of $\pi$ remembers the configuration of the last level-$1$ marked point, which precisely determines the complex structure of the curve.

\

{\bf Step 1: Generating and distributing $\ell_2+1$.}

Define
\[f_{\ell}(x) = \frac18 \sum_{d,\bar{\ell}} \frac{Q^d}{\bar{\ell}!} str_h H^\ast\left(\pi^{-1}(\{C\}), \prod_{j=1,2,3} ev_j^\ast(\bar{x}_{1,-1}) \cdot ev_4^\ast (x) \cdot \prod_{i=1}^\ell ev_{2i+3}^\ast (\bar{t}_{2,1})ev_{2i+4}^\ast (\bar{t}_{2,1}) \prod ev^\ast \tau \right),\]
and let 
\[\delta(x) = \sum_{\alpha,\beta} \phi_\alpha G_1^{\alpha\beta} \<\!\<\phi_\beta, \bar{t}_{2,1},x\>\!\>_{0,1_1+1_2+1_1}.\]
\begin{lemma}
\[\begin{array}{ll}
&\dfrac1{2^{\ell+1}(\ell+1)!}\<\!\<\bar{x}_{1,-1},\bar{x}_{1,-1},\bar{x}_{1,-1},x;\bar{t}_{2,1}(\bar{L}-1),\cdots\>\!\>_{1,4_1+(\ell+1)_2} \\
=&f_{\ell+1}(x)+ \dfrac1{2^{\ell}\ell!} \<\!\<\bar{x}_{1,-1},\bar{x}_{1,-1},\bar{x}_{1,-1},\delta(x);\bar{t}_{2,1}(\bar{L}-1),\cdots\>\!\>_{1,4_1+(\ell)_2}.
\end{array}\]
\end{lemma}

\begin{proof}
The dimension
\[\dim\M_{1,4+2\ell} - \dim\M^2_\ell = \ell+3.\]

By definition of invariants and Section 1.3, we have
\[\begin{array}{ll}
&\<\!\<\bar{x}_{1,-1},\bar{x}_{1,-1},\bar{x}_{1,-1},\delta(x);\bar{t}_{2,1}(\bar{L}-1),\cdots\>\!\>_{1,4_1+(\ell)_2}\\
=&str_h H^\ast \left(\M_{1,4+2\ell_2+|\bar{\ell}|} (X,d), V_{2c}\right)\\
=&\dfrac1{2^{\ell+3}}str_h H^\ast \left(\M^2_\ell(X,d), V_{2c}\right)
\end{array}\]
for some appropriate $V_{2c}$.

At the level-2 marked points, consider the product $L_{2i+3} L_{2i+4}$ of the universal cotangent bundles on $\mathcal{M}_\ell^2$. 
Away from the divisor $D_i$ (defined by the $(i+3)$-rd level-1 marked point), this product coincides with the square of $L_4$ pulled back along $\pi$.
On $D_i$, however, the product $L_{2i+3} L_{2i+4}$ is trivial, while the pullback of $L_4^2$ is $2\mathcal{O}(-D_i)$. 
Consequently, the difference between these two sheaves is $2 i_\ast \mathcal{O}_{D_i}$, which is the push forward of the structure sheaf of $D_i$ under the inclusion $i: D_i \hookrightarrow \mathcal{M}_{\ell_2}$. 
From this, we deduce
\[L_{2i+3} L_{2i+4} - 1 = 2 \left( \pi^* (L_{i+3} - 1) + i_\ast \mathcal{O}_{D_i} \right) + \mathcal{O}((-L_{i+3} - 1)^2).\]
The last term does not contribute to correlators by a dimension argument.
Applying the above formula to all level-$2$ marked points, the result expands into $2^{\ell_2} \ell_2!$ terms.
Since the divisors $D_i$ are pairwise disjoint, all but the following two types of terms vanish by the (cohomological) Dilaton equation:

\begin{enumerate}
\item
$\dfrac{1}{2^{\ell} \ell!} \dprod_i \pi^* (L_{i+3} - 1)$.
\item
$\dfrac{1}{2^{\ell} \ell!} \dsum_i i_* D_i \cdot \dprod_2 \pi^* (L_{i+3} - 1)$.
\end{enumerate}

Notice that for each $i$, there are two ways to label marked points $2i+3,2i+4$, contributing a factor $2^{\ell+1}$.
The first type contributes $f_{\ell+1}(x)$.
For the second type, the splitting axiom implies that the contributions from maps with these base curves coincide, contributing the second term on the right-hand side of the lemma.

\tikzset{every picture/.style={line width=0.75pt}}
\begin{tikzpicture}[x=0.75pt,y=0.75pt,yscale=-1,xscale=1]
\draw  [draw opacity=0] (101.89,15.06) .. controls (102.26,15.04) and (102.63,15.04) .. (103,15.03) .. controls (119.57,14.97) and (133.05,27.44) .. (133.12,42.89) .. controls (133.18,58.35) and (119.8,70.93) .. (103.23,71) .. controls (103.18,71) and (103.12,71) .. (103.06,71) -- (103.12,43.02) -- cycle ; \draw   (101.89,15.06) .. controls (102.26,15.04) and (102.63,15.04) .. (103,15.03) .. controls (119.57,14.97) and (133.05,27.44) .. (133.12,42.89) .. controls (133.18,58.35) and (119.8,70.93) .. (103.23,71) .. controls (103.18,71) and (103.12,71) .. (103.06,71) ;  
\draw  [color={rgb, 255:red, 74; green, 144; blue, 226 }  ,draw opacity=1 ][line width=3]  (180.93,43.16) .. controls (180.94,43.95) and (180.29,44.6) .. (179.48,44.6) .. controls (178.67,44.6) and (178.01,43.96) .. (178.01,43.17) .. controls (178.01,42.37) and (178.66,41.73) .. (179.47,41.72) .. controls (180.27,41.72) and (180.93,42.36) .. (180.93,43.16) -- cycle ;
\draw  [color={rgb, 255:red, 208; green, 2; blue, 27 }  ,draw opacity=1 ][line width=3]  (158.25,41.51) .. controls (158.25,42.31) and (157.6,42.95) .. (156.79,42.96) .. controls (155.98,42.96) and (155.33,42.32) .. (155.32,41.52) .. controls (155.32,40.73) and (155.97,40.08) .. (156.78,40.08) .. controls (157.59,40.08) and (158.24,40.72) .. (158.25,41.51) -- cycle ;
\draw  [draw opacity=0] (132.13,42.43) .. controls (138.2,36.45) and (146.68,32.72) .. (156.08,32.67) .. controls (165.81,32.63) and (174.59,36.55) .. (180.75,42.85) -- (156.22,64.93) -- cycle ; \draw   (132.13,42.43) .. controls (138.2,36.45) and (146.68,32.72) .. (156.08,32.67) .. controls (165.81,32.63) and (174.59,36.55) .. (180.75,42.85) ;  
\draw  [draw opacity=0] (180.75,43.23) .. controls (174.71,48.93) and (166.43,52.47) .. (157.28,52.51) .. controls (147.41,52.55) and (138.51,48.52) .. (132.33,42.05) -- (157.14,20.25) -- cycle ; \draw   (180.75,43.23) .. controls (174.71,48.93) and (166.43,52.47) .. (157.28,52.51) .. controls (147.41,52.55) and (138.51,48.52) .. (132.33,42.05) ;  
\draw  [draw opacity=0] (293.87,10.99) .. controls (313.34,11.1) and (329.12,25.55) .. (329.2,43.42) .. controls (329.27,61.41) and (313.4,76.07) .. (293.74,76.15) .. controls (293.68,76.15) and (293.61,76.15) .. (293.54,76.15) -- (293.61,43.57) -- cycle ; \draw   (293.87,10.99) .. controls (313.34,11.1) and (329.12,25.55) .. (329.2,43.42) .. controls (329.27,61.41) and (313.4,76.07) .. (293.74,76.15) .. controls (293.68,76.15) and (293.61,76.15) .. (293.54,76.15) ;  
\draw  [color={rgb, 255:red, 74; green, 144; blue, 226 }  ,draw opacity=1 ][line width=3]  (329.93,43.16) .. controls (329.94,43.95) and (329.29,44.6) .. (328.48,44.6) .. controls (327.67,44.6) and (327.01,43.96) .. (327.01,43.17) .. controls (327.01,42.37) and (327.66,41.73) .. (328.47,41.72) .. controls (329.27,41.72) and (329.93,42.36) .. (329.93,43.16) -- cycle ;
\draw (87,33.4) node [anchor=north west][inner sep=0.75pt]    {$\cdots $};
\draw (222,34.4) node [anchor=north west][inner sep=0.75pt]    {$=$};
\draw (140,7.4) node [anchor=north west][inner sep=0.75pt]    {$\overline{t}_{2,1}$};
\draw (193,33.4) node [anchor=north west][inner sep=0.75pt]    {$x$};
\draw (273,32.4) node [anchor=north west][inner sep=0.75pt]    {$\cdots $};
\draw (342,36.4) node [anchor=north west][inner sep=0.75pt]    {$\delta ( x)$};
\end{tikzpicture}
\end{proof}

Inductively applying the lemma yields
\[\dfrac1{2^{\ell}\ell!} \<\!\<\bar{x}_{1,-1},\bar{x}_{1,-1},\bar{x}_{1,-1},\delta(x);\bar{t}_{2,1}(\bar{L}-1),\cdots\>\!\>_{1,4_1+(\ell)_2} = \sum_{n=0}^\ell f_{\ell-n}(\delta^n(x)).\]
By the splitting axiom interpretation of $\delta$, all of $f_{\ell-n}(\delta^n(x))$ are equal to $f_\ell(x)$.

An argument identical to Case $2a$ shows that summing over all $\ell$ gives the total contribution
\[\frac1{24}\<\!\<\bar{x}_{1,-1}, \bar{y}_{1,-1}, \bar{y}_{1,-1}, \bar{x}_{1,-1}\>\!\>_{1,4_1}.\]

\

{\bf Step 2: Modifications}

Next, observe that contributions from fibers of $ft_i: \M_{\ell_2+1}^2(X,d) \to \M_1^2 \to \M_{0,4}$ forgetting all but the $i$-th pair of marked points do not depend on the specific fiber.
Consequently, the contributions of the following curves agree.

\tikzset{every picture/.style={line width=0.75pt}}
\begin{tikzpicture}[x=0.7pt,y=0.7pt,yscale=-1,xscale=1]
\draw   (175.67,19.36) -- (223.18,19.29) -- (223.24,60.14) -- (175.73,60.21) -- cycle ;
\draw  [color={rgb, 255:red, 74; green, 144; blue, 226 }  ,draw opacity=1 ][line width=3]  (221.78,60.14) .. controls (221.78,59.34) and (222.43,58.7) .. (223.24,58.7) .. controls (224.05,58.7) and (224.7,59.34) .. (224.7,60.14) .. controls (224.7,60.93) and (224.05,61.58) .. (223.24,61.58) .. controls (222.43,61.58) and (221.78,60.93) .. (221.78,60.14) -- cycle ;
\draw  [color={rgb, 255:red, 74; green, 144; blue, 226 }  ,draw opacity=1 ][line width=3]  (127.43,17.92) .. controls (127.43,17.13) and (128.08,16.49) .. (128.89,16.49) .. controls (129.7,16.49) and (130.35,17.13) .. (130.35,17.92) .. controls (130.35,18.72) and (129.7,19.36) .. (128.89,19.36) .. controls (128.08,19.36) and (127.43,18.72) .. (127.43,17.92) -- cycle ;
\draw  [color={rgb, 255:red, 74; green, 144; blue, 226 }  ,draw opacity=1 ][line width=3]  (221.71,19.29) .. controls (221.71,18.49) and (222.37,17.85) .. (223.17,17.85) .. controls (223.98,17.85) and (224.64,18.49) .. (224.64,19.29) .. controls (224.64,20.08) and (223.98,20.73) .. (223.17,20.73) .. controls (222.37,20.73) and (221.71,20.08) .. (221.71,19.29) -- cycle ;
\draw  [color={rgb, 255:red, 208; green, 2; blue, 27 }  ,draw opacity=1 ][line width=3]  (150.11,19.67) .. controls (150.11,18.87) and (150.76,18.23) .. (151.57,18.23) .. controls (152.38,18.23) and (153.03,18.87) .. (153.03,19.67) .. controls (153.03,20.46) and (152.38,21.11) .. (151.57,21.11) .. controls (150.76,21.11) and (150.11,20.46) .. (150.11,19.67) -- cycle ;
\draw  [draw opacity=0] (176.23,18.87) .. controls (170.13,24.82) and (161.63,28.51) .. (152.24,28.51) .. controls (142.51,28.51) and (133.75,24.56) .. (127.6,18.23) -- (152.24,-3.74) -- cycle ; \draw   (176.23,18.87) .. controls (170.13,24.82) and (161.63,28.51) .. (152.24,28.51) .. controls (142.51,28.51) and (133.75,24.56) .. (127.6,18.23) ;  
\draw  [draw opacity=0] (127.61,17.85) .. controls (133.68,12.17) and (141.97,8.67) .. (151.12,8.67) .. controls (160.99,8.67) and (169.87,12.75) .. (176.02,19.24) -- (151.12,40.93) -- cycle ; \draw   (127.61,17.85) .. controls (133.68,12.17) and (141.97,8.67) .. (151.12,8.67) .. controls (160.99,8.67) and (169.87,12.75) .. (176.02,19.24) ;  
\draw    (223.17,19.29) .. controls (252.41,-2.29) and (267.03,40.87) .. (296.27,19.29) ;
\draw    (223.24,60.14) .. controls (252.48,38.56) and (267.1,81.71) .. (296.33,60.14) ;
\draw  [color={rgb, 255:red, 74; green, 144; blue, 226 }  ,draw opacity=1 ][line width=3]  (174.4,107.29) .. controls (173.59,107.29) and (172.94,106.65) .. (172.94,105.85) .. controls (172.94,105.06) and (173.59,104.41) .. (174.4,104.41) .. controls (175.2,104.41) and (175.86,105.06) .. (175.86,105.85) .. controls (175.86,106.65) and (175.21,107.29) .. (174.4,107.29) -- cycle ;
\draw  [color={rgb, 255:red, 208; green, 2; blue, 27 }  ,draw opacity=1 ][line width=3]  (146.45,84.93) .. controls (145.65,84.93) and (144.99,84.29) .. (144.99,83.5) .. controls (144.99,82.7) and (145.64,82.06) .. (146.45,82.06) .. controls (147.26,82.06) and (147.91,82.7) .. (147.91,83.5) .. controls (147.91,84.29) and (147.26,84.93) .. (146.45,84.93) -- cycle ;
\draw  [draw opacity=0] (174.96,58.9) .. controls (181.23,64.92) and (185.15,73.44) .. (185.15,82.87) .. controls (185.15,92.64) and (180.97,101.41) .. (174.32,107.48) -- (152.37,82.89) -- cycle ; \draw   (174.96,58.9) .. controls (181.23,64.92) and (185.15,73.44) .. (185.15,82.87) .. controls (185.15,92.64) and (180.97,101.41) .. (174.32,107.48) ;  
\draw  [draw opacity=0] (174.7,107.49) .. controls (168.7,101.49) and (164.99,93.17) .. (164.99,83.98) .. controls (164.98,74.08) and (169.29,65.19) .. (176.11,59.12) -- (197.77,83.97) -- cycle ; \draw   (174.7,107.49) .. controls (168.7,101.49) and (164.99,93.17) .. (164.99,83.98) .. controls (164.98,74.08) and (169.29,65.19) .. (176.11,59.12) ;  
\draw   (128.18,83.5) .. controls (128.18,73.56) and (136.36,65.51) .. (146.45,65.51) .. controls (156.54,65.51) and (164.73,73.56) .. (164.73,83.5) .. controls (164.73,93.43) and (156.54,101.48) .. (146.45,101.48) .. controls (136.36,101.48) and (128.18,93.43) .. (128.18,83.5) -- cycle ;
\draw  [color={rgb, 255:red, 208; green, 2; blue, 27 }  ,draw opacity=1 ][line width=3]  (75.55,87.09) .. controls (74.74,87.09) and (74.09,86.45) .. (74.09,85.65) .. controls (74.09,84.86) and (74.74,84.22) .. (75.55,84.21) .. controls (76.36,84.21) and (77.01,84.86) .. (77.01,85.65) .. controls (77.01,86.45) and (76.36,87.09) .. (75.55,87.09) -- cycle ;
\draw   (57.28,85.65) .. controls (57.28,75.72) and (65.46,67.67) .. (75.55,67.67) .. controls (85.64,67.67) and (93.82,75.72) .. (93.82,85.65) .. controls (93.82,95.58) and (85.64,103.63) .. (75.55,103.63) .. controls (65.46,103.63) and (57.28,95.58) .. (57.28,85.65) -- cycle ;
\draw  [color={rgb, 255:red, 208; green, 2; blue, 27 }  ,draw opacity=1 ][line width=3]  (29.5,86.37) .. controls (28.7,86.37) and (28.04,85.73) .. (28.04,84.93) .. controls (28.04,84.14) and (28.69,83.5) .. (29.5,83.5) .. controls (30.31,83.5) and (30.96,84.14) .. (30.96,84.93) .. controls (30.96,85.73) and (30.31,86.37) .. (29.5,86.37) -- cycle ;
\draw   (20,85.65) .. controls (20,75.72) and (28.18,67.67) .. (38.27,67.67) .. controls (48.37,67.67) and (56.55,75.72) .. (56.55,85.65) .. controls (56.55,95.58) and (48.37,103.63) .. (38.27,103.63) .. controls (28.18,103.63) and (20,95.58) .. (20,85.65) -- cycle ;
\draw  [color={rgb, 255:red, 208; green, 2; blue, 27 }  ,draw opacity=1 ][line width=3]  (47.05,87.09) .. controls (46.24,87.09) and (45.58,86.45) .. (45.58,85.65) .. controls (45.58,84.86) and (46.24,84.22) .. (47.04,84.21) .. controls (47.85,84.21) and (48.51,84.86) .. (48.51,85.65) .. controls (48.51,86.45) and (47.85,87.09) .. (47.05,87.09) -- cycle ;
\draw   (518.57,20.36) -- (566.08,20.29) -- (566.14,61.14) -- (518.64,61.21) -- cycle ;
\draw  [color={rgb, 255:red, 74; green, 144; blue, 226 }  ,draw opacity=1 ][line width=3]  (564.68,61.14) .. controls (564.68,60.34) and (565.34,59.7) .. (566.14,59.7) .. controls (566.95,59.7) and (567.6,60.34) .. (567.6,61.14) .. controls (567.6,61.93) and (566.95,62.58) .. (566.14,62.58) .. controls (565.34,62.58) and (564.68,61.93) .. (564.68,61.14) -- cycle ;
\draw  [color={rgb, 255:red, 74; green, 144; blue, 226 }  ,draw opacity=1 ][line width=3]  (564.62,20.29) .. controls (564.62,19.49) and (565.27,18.85) .. (566.08,18.85) .. controls (566.89,18.85) and (567.54,19.49) .. (567.54,20.29) .. controls (567.54,21.08) and (566.89,21.73) .. (566.08,21.73) .. controls (565.27,21.73) and (564.62,21.08) .. (564.62,20.29) -- cycle ;
\draw    (566.08,20.29) .. controls (595.32,-1.29) and (609.93,41.87) .. (639.17,20.29) ;
\draw    (566.14,61.14) .. controls (595.38,39.56) and (610,82.71) .. (639.24,61.14) ;
\draw  [color={rgb, 255:red, 74; green, 144; blue, 226 }  ,draw opacity=1 ][line width=3]  (517.3,108.29) .. controls (516.49,108.29) and (515.84,107.65) .. (515.84,106.85) .. controls (515.84,106.06) and (516.49,105.41) .. (517.3,105.41) .. controls (518.11,105.41) and (518.76,106.06) .. (518.76,106.85) .. controls (518.76,107.65) and (518.11,108.29) .. (517.3,108.29) -- cycle ;
\draw  [color={rgb, 255:red, 208; green, 2; blue, 27 }  ,draw opacity=1 ][line width=3]  (489.36,85.93) .. controls (488.55,85.93) and (487.89,85.29) .. (487.89,84.5) .. controls (487.89,83.7) and (488.55,83.06) .. (489.36,83.06) .. controls (490.16,83.06) and (490.82,83.7) .. (490.82,84.5) .. controls (490.82,85.29) and (490.16,85.93) .. (489.36,85.93) -- cycle ;
\draw  [draw opacity=0] (517.86,59.9) .. controls (524.14,65.92) and (528.05,74.44) .. (528.05,83.87) .. controls (528.06,93.64) and (523.88,102.41) .. (517.22,108.48) -- (495.27,83.89) -- cycle ; \draw   (517.86,59.9) .. controls (524.14,65.92) and (528.05,74.44) .. (528.05,83.87) .. controls (528.06,93.64) and (523.88,102.41) .. (517.22,108.48) ;  
\draw  [draw opacity=0] (517.6,108.49) .. controls (511.61,102.49) and (507.89,94.17) .. (507.89,84.98) .. controls (507.89,75.08) and (512.19,66.19) .. (519.02,60.12) -- (540.67,84.97) -- cycle ; \draw   (517.6,108.49) .. controls (511.61,102.49) and (507.89,94.17) .. (507.89,84.98) .. controls (507.89,75.08) and (512.19,66.19) .. (519.02,60.12) ;  
\draw   (471.08,84.5) .. controls (471.08,74.56) and (479.26,66.51) .. (489.36,66.51) .. controls (499.45,66.51) and (507.63,74.56) .. (507.63,84.5) .. controls (507.63,94.43) and (499.45,102.48) .. (489.36,102.48) .. controls (479.26,102.48) and (471.08,94.43) .. (471.08,84.5) -- cycle ;
\draw  [color={rgb, 255:red, 208; green, 2; blue, 27 }  ,draw opacity=1 ][line width=3]  (418.46,88.09) .. controls (417.65,88.09) and (416.99,87.45) .. (416.99,86.65) .. controls (416.99,85.86) and (417.65,85.22) .. (418.45,85.21) .. controls (419.26,85.21) and (419.92,85.86) .. (419.92,86.65) .. controls (419.92,87.45) and (419.26,88.09) .. (418.46,88.09) -- cycle ;
\draw   (400.18,86.65) .. controls (400.18,76.72) and (408.36,68.67) .. (418.45,68.67) .. controls (428.55,68.67) and (436.73,76.72) .. (436.73,86.65) .. controls (436.73,96.58) and (428.55,104.63) .. (418.45,104.63) .. controls (408.36,104.63) and (400.18,96.58) .. (400.18,86.65) -- cycle ;
\draw  [color={rgb, 255:red, 208; green, 2; blue, 27 }  ,draw opacity=1 ][line width=3]  (372.41,87.37) .. controls (371.6,87.37) and (370.94,86.73) .. (370.94,85.93) .. controls (370.94,85.14) and (371.6,84.5) .. (372.41,84.5) .. controls (373.21,84.5) and (373.87,85.14) .. (373.87,85.93) .. controls (373.87,86.73) and (373.21,87.37) .. (372.41,87.37) -- cycle ;
\draw   (362.9,86.65) .. controls (362.9,76.72) and (371.08,68.67) .. (381.18,68.67) .. controls (391.27,68.67) and (399.45,76.72) .. (399.45,86.65) .. controls (399.45,96.58) and (391.27,104.63) .. (381.18,104.63) .. controls (371.08,104.63) and (362.9,96.58) .. (362.9,86.65) -- cycle ;
\draw  [color={rgb, 255:red, 208; green, 2; blue, 27 }  ,draw opacity=1 ][line width=3]  (389.95,88.09) .. controls (389.14,88.09) and (388.49,87.45) .. (388.49,86.65) .. controls (388.49,85.86) and (389.14,85.22) .. (389.95,85.21) .. controls (390.76,85.21) and (391.41,85.86) .. (391.41,86.65) .. controls (391.41,87.45) and (390.76,88.09) .. (389.95,88.09) -- cycle ;
\draw  [color={rgb, 255:red, 74; green, 144; blue, 226 }  ,draw opacity=1 ][line width=3]  (518.57,20.36) .. controls (517.77,20.36) and (517.11,19.72) .. (517.11,18.92) .. controls (517.11,18.13) and (517.76,17.49) .. (518.57,17.49) .. controls (519.38,17.49) and (520.03,18.13) .. (520.03,18.92) .. controls (520.03,19.72) and (519.38,20.36) .. (518.57,20.36) -- cycle ;

\draw (93.58,73.02) node [anchor=north west][inner sep=0.75pt]    {$\cdots $};
\draw (33.34,120.49) node [anchor=north west][inner sep=0.75pt]    {$\ell -1$};
\draw (20.1,108.49) node [anchor=north west][inner sep=0.75pt]    {$\underbrace{\ \ \ \ \ \ \ \ \ \ \ \ \ \ \ \ \ \ \ \ \ \ \ \ \ \ }$};
\draw (73.85,120.81) node [anchor=north west][inner sep=0.75pt]   [align=left] {red points};
\draw (436.48,74.02) node [anchor=north west][inner sep=0.75pt]    {$\cdots $};
\draw (399.24,121.49) node [anchor=north west][inner sep=0.75pt]    {$\ell $};
\draw (363,109.49) node [anchor=north west][inner sep=0.75pt]    {$\underbrace{\ \ \ \ \ \ \ \ \ \ \ \ \ \ \ \ \ \ \ \ \ \ \ \ \ }$};
\draw (416.75,121.81) node [anchor=north west][inner sep=0.75pt]   [align=left] {red points};
\draw (310,74.4) node [anchor=north west][inner sep=0.75pt]    {$=$};

\end{tikzpicture}

Hence, the following two types of curves have equal contributions:

\begin{tikzpicture}[x=0.5pt,y=0.5pt,yscale=-1,xscale=1]
\draw   (7.74,70) -- (39.71,23.18) -- (86.61,55.21) -- (54.64,102.03) -- cycle ;
\draw  [color={rgb, 255:red, 74; green, 144; blue, 226 }  ,draw opacity=1 ][line width=3]  (38.58,24.83) .. controls (37.67,24.21) and (37.43,22.96) .. (38.06,22.05) .. controls (38.68,21.14) and (39.93,20.91) .. (40.84,21.53) .. controls (41.75,22.16) and (41.98,23.4) .. (41.36,24.31) .. controls (40.73,25.22) and (39.49,25.46) .. (38.58,24.83) -- cycle ;
\draw  [color={rgb, 255:red, 74; green, 144; blue, 226 }  ,draw opacity=1 ][line width=3]  (85.48,56.86) .. controls (84.57,56.24) and (84.33,54.99) .. (84.96,54.08) .. controls (85.58,53.17) and (86.83,52.94) .. (87.74,53.56) .. controls (88.65,54.19) and (88.88,55.43) .. (88.26,56.34) .. controls (87.63,57.25) and (86.39,57.49) .. (85.48,56.86) -- cycle ;
\draw  [color={rgb, 255:red, 126; green, 211; blue, 33 }  ,draw opacity=1 ][line width=3]  (6.6,71.65) .. controls (5.69,71.02) and (5.46,69.78) .. (6.09,68.87) .. controls (6.71,67.96) and (7.95,67.72) .. (8.87,68.35) .. controls (9.78,68.97) and (10.01,70.22) .. (9.39,71.13) .. controls (8.76,72.04) and (7.52,72.27) .. (6.6,71.65) -- cycle ;
\draw  [color={rgb, 255:red, 74; green, 144; blue, 226 }  ,draw opacity=1 ][line width=3]  (53.51,103.68) .. controls (52.59,103.05) and (52.36,101.81) .. (52.99,100.9) .. controls (53.61,99.99) and (54.86,99.75) .. (55.77,100.38) .. controls (56.68,101) and (56.91,102.25) .. (56.29,103.16) .. controls (55.66,104.07) and (54.42,104.3) .. (53.51,103.68) -- cycle ;
\draw    (40.84,21.53) .. controls (80.84,-8.47) and (100.84,51.53) .. (140.84,21.53) ;
\draw    (86.61,55.21) .. controls (126.61,25.21) and (146.61,85.21) .. (186.61,55.21) ;
\draw   (213.74,74) -- (245.71,27.18) -- (292.61,59.21) -- (260.64,106.03) -- cycle ;
\draw  [color={rgb, 255:red, 74; green, 144; blue, 226 }  ,draw opacity=1 ][line width=3]  (244.58,28.83) .. controls (243.67,28.21) and (243.43,26.96) .. (244.06,26.05) .. controls (244.68,25.14) and (245.93,24.91) .. (246.84,25.53) .. controls (247.75,26.16) and (247.98,27.4) .. (247.36,28.31) .. controls (246.73,29.22) and (245.49,29.46) .. (244.58,28.83) -- cycle ;
\draw  [color={rgb, 255:red, 74; green, 144; blue, 226 }  ,draw opacity=1 ][line width=3]  (291.48,60.86) .. controls (290.57,60.24) and (290.33,58.99) .. (290.96,58.08) .. controls (291.58,57.17) and (292.83,56.94) .. (293.74,57.56) .. controls (294.65,58.19) and (294.88,59.43) .. (294.26,60.34) .. controls (293.63,61.25) and (292.39,61.49) .. (291.48,60.86) -- cycle ;
\draw  [color={rgb, 255:red, 126; green, 211; blue, 33 }  ,draw opacity=1 ][line width=3]  (212.6,75.65) .. controls (211.69,75.02) and (211.46,73.78) .. (212.09,72.87) .. controls (212.71,71.96) and (213.95,71.72) .. (214.87,72.35) .. controls (215.78,72.97) and (216.01,74.22) .. (215.39,75.13) .. controls (214.76,76.04) and (213.52,76.27) .. (212.6,75.65) -- cycle ;
\draw  [color={rgb, 255:red, 74; green, 144; blue, 226 }  ,draw opacity=1 ][line width=3]  (259.51,107.68) .. controls (258.59,107.05) and (258.36,105.81) .. (258.99,104.9) .. controls (259.61,103.99) and (260.86,103.75) .. (261.77,104.38) .. controls (262.68,105) and (262.91,106.25) .. (262.29,107.16) .. controls (261.66,108.07) and (260.42,108.3) .. (259.51,107.68) -- cycle ;
\draw    (246.84,25.53) .. controls (286.84,-4.47) and (306.84,55.53) .. (346.84,25.53) ;
\draw    (292.61,59.21) .. controls (332.61,29.21) and (352.61,89.21) .. (392.61,59.21) ;
\draw    (260.64,106.03) .. controls (300.64,76.03) and (320.64,136.03) .. (360.64,106.03) ;
\draw   (477.74,76.12) -- (509.71,29.3) -- (556.61,61.33) -- (524.64,108.15) -- cycle ;
\draw  [color={rgb, 255:red, 74; green, 144; blue, 226 }  ,draw opacity=1 ][line width=3]  (508.58,30.95) .. controls (507.67,30.33) and (507.43,29.08) .. (508.06,28.17) .. controls (508.68,27.26) and (509.93,27.03) .. (510.84,27.65) .. controls (511.75,28.28) and (511.98,29.52) .. (511.36,30.43) .. controls (510.73,31.34) and (509.49,31.58) .. (508.58,30.95) -- cycle ;
\draw  [color={rgb, 255:red, 74; green, 144; blue, 226 }  ,draw opacity=1 ][line width=3]  (555.48,62.98) .. controls (554.57,62.36) and (554.33,61.11) .. (554.96,60.2) .. controls (555.58,59.29) and (556.83,59.06) .. (557.74,59.68) .. controls (558.65,60.31) and (558.88,61.55) .. (558.26,62.46) .. controls (557.63,63.37) and (556.39,63.61) .. (555.48,62.98) -- cycle ;
\draw  [color={rgb, 255:red, 126; green, 211; blue, 33 }  ,draw opacity=1 ][line width=3]  (427.54,78.31) .. controls (426.63,77.69) and (426.4,76.45) .. (427.02,75.53) .. controls (427.65,74.62) and (428.89,74.39) .. (429.8,75.01) .. controls (430.71,75.64) and (430.95,76.88) .. (430.32,77.8) .. controls (429.7,78.71) and (428.45,78.94) .. (427.54,78.31) -- cycle ;
\draw  [color={rgb, 255:red, 74; green, 144; blue, 226 }  ,draw opacity=1 ][line width=3]  (523.51,109.8) .. controls (522.59,109.17) and (522.36,107.93) .. (522.99,107.02) .. controls (523.61,106.11) and (524.86,105.87) .. (525.77,106.5) .. controls (526.68,107.12) and (526.91,108.37) .. (526.29,109.28) .. controls (525.66,110.19) and (524.42,110.42) .. (523.51,109.8) -- cycle ;
\draw    (510.84,27.65) .. controls (550.84,-2.35) and (570.84,57.65) .. (610.84,27.65) ;
\draw    (556.61,61.33) .. controls (596.61,31.33) and (616.61,91.33) .. (656.61,61.33) ;
\draw    (524.64,108.15) .. controls (564.64,78.15) and (584.64,138.15) .. (624.64,108.15) ;
\draw  [draw opacity=0] (427.02,75.53) .. controls (432.37,67.3) and (441.65,61.86) .. (452.19,61.86) .. controls (462.99,61.86) and (472.45,67.56) .. (477.74,76.12) -- (452.19,91.86) -- cycle ; \draw   (427.02,75.53) .. controls (432.37,67.3) and (441.65,61.86) .. (452.19,61.86) .. controls (462.99,61.86) and (472.45,67.56) .. (477.74,76.12) ;  
\draw  [draw opacity=0] (479.04,75.5) .. controls (473.86,84.5) and (464.15,90.55) .. (453.02,90.55) .. controls (441.91,90.55) and (432.21,84.51) .. (427.02,75.53) -- (453.02,60.55) -- cycle ; \draw   (479.04,75.5) .. controls (473.86,84.5) and (464.15,90.55) .. (453.02,90.55) .. controls (441.91,90.55) and (432.21,84.51) .. (427.02,75.53) ;  
\draw  [color={rgb, 255:red, 208; green, 2; blue, 27 }  ,draw opacity=1 ][line width=3]  (452.6,78.65) .. controls (451.69,78.02) and (451.46,76.78) .. (452.09,75.87) .. controls (452.71,74.96) and (453.95,74.72) .. (454.87,75.35) .. controls (455.78,75.97) and (456.01,77.22) .. (455.39,78.13) .. controls (454.76,79.04) and (453.52,79.27) .. (452.6,78.65) -- cycle ;
\draw (181,64.4) node [anchor=north west][inner sep=0.75pt]    {$=$};
\draw (401,60.4) node [anchor=north west][inner sep=0.75pt]    {$-$};
\end{tikzpicture}

Thus, the total contribution rewrites as
\[\frac1{24}\<\!\<\bar{x}^{2c}_{1,-1}, \bar{y}_{1,-1}, \bar{y}_{1,-1}, \bar{y}_{1,-1}\>\!\>_{1,4_1},\]
where
\[\bar{x}^{2c}_{1,-1} = \bar{x}_{1,-1} - \sum_{\alpha,\beta} \phi_\alpha G_1^{\alpha\beta}\<\!\<\phi_\beta,\bar{t}_{2,1},\bar{x}_{1,-1}\>\!\>_{0,1_1+1_2+1_1}.\]

\subsection{Total contribution in Case 2}
\

\begin{lemma}
\[\begin{array}{ll}
&\<\!\<(\bar{x}^{2a}_{1,-1}+\bar{x}^{2b}_{1,-1})(-\bar{L}-1)+\bar{x}^{2c}_{1,-1}, \bar{y}_{1,-1}, \bar{y}_{1,-1}, \bar{y}_{1,-1}\>\!\>_{1,4_1} \\
=& Res_{-1} \<\!\<\dfrac{\y^L_2(q)}{1-q^{-1}L}, \y_2(q), \y_2(q), \y_2(q)\>\!\>_{1,4_1}\dfrac{dq}q.
\end{array}\]
\end{lemma}

\begin{proof}
For the right-hand side, recall that
\[\begin{array}{ll}
\y_2(q)&= \bar{\x}_1(q)-\dsum_{\alpha,\beta}\phi_\alpha G_1^{\alpha\beta}\<\!\<\phi_\beta,\dfrac{\partial \tau_2}{\partial t_{2,0}^1}-1, \bar{\x}_1(q)\>\!\>_{0,1_1+1_2+1_1}\\
\y^L_2(q)&= \bar{\x}_1(q)-\dsum_{\alpha,\beta}\phi_\alpha G_1^{\alpha\beta}\<\!\<\phi_\beta,\bar{\t}^{new}_2(q), \bar{\x}_1(q)\>\!\>_{0,1_1+1_2+1_1}\\
\bar{\t}^{new}_2(q-1) & = -\bar{\t}_2(q+1)/q.
\end{array}\]
Expanding $\bar{\t}_2^{new}(q)$ at $q = -1$ then gives 
\[\bar{\t}_2^{new}(q) = \bar{t}_{2,1} + \bar{t}_{2,1}' (-q - 1) + O((-q - 1)^2).\]

One check
\[\begin{array}{llll}
\y_2(q)& = \bar{y}_{1,-1} & + '\bar{x}_{1,-1}^{2a} (-q-1) & + O((-q-1)^2),\\
\y^L_2(q)& = \bar{x}_{1,-1}^{2c} & + (\bar{x}_{1,-1}^{2a}+\bar{x}_{1,-1}^{2b}) (-q-1) & + O((-q-1)^2).
\end{array}\]

As in \cite{Tang2}, when we multiply the first input by $1/(1 - q^{-1}L)$ and then take the corresponding residue, the effect is to replace each occurrence of $-q - 1$, arising from the products of $\bar{\x}$'s, by $-L - 1$ in the first input.
By dimension argument, everything with $(-L-1)^2$ vanishes, and it is straightforward to check that both sides coincide, using
\[\<\!\<\bar{x}^{2a}_{1,-1}(-\bar{L}-1), \bar{y}_{1,-1}, \bar{y}_{1,-1}, \bar{y}_{1,-1}\>\!\>_{1,4_1}
= \<\!\<'\bar{x}^{2a}_{1,-1}(-\bar{L}-1), \bar{y}'_{1,-1}, \bar{y}_{1,-1}, \bar{y}_{1,-1}\>\!\>_{1,4_1}\]
\end{proof}

Hence, the total contribution is 
\[\begin{array}{ll}
&Res_{-1} \dfrac1{24} \<\!\<\dfrac{\y^L_2(q)}{1-q^{-1}L}, \y_2(q), \y_2(q), \y_2(q)\>\!\>_{1,4_1}\dfrac{dq}q\\
=&Res_{0,\infty} \dfrac1{24} \<\!\<\dfrac{\y^L_2(x^{-1})}{1-xL}, \y_2(x^{-1}), \y_2(x^{-1}), \y_2(x^{-1})\>\!\>_{1,4_1}\dfrac{dx}x.
\end{array}\]

\subsection{Cases 3,4,5}
\

In Case $3\pm$, the non-vanishing correlators are
\[\begin{array}{ll}
&\dsum_{\ell_4 \geq 0} \dfrac1{2!1!\ell_4!} \dsum_{\bar{\ell}_4,d} \frac{Q^d}{\bar{\ell}_4!} \cdot \left(\sum_{\pm} \<\bar{x}_{1,\pm i},\bar{x}_{1,\pm i};\bar{x}_{2,-1};\bar{t}_{4,1}(\bar{L}-1),\cdots ;\tau_4, \cdots \>_{1,2_1+1_2+(\ell_4)_4+(\bar{\ell}_4)_4,d}\right)\\
=& \dsum_{\ell_4 \geq 0, \pm}  \dfrac1{c_0(str_{h_4} N^\ast_{\M_{1,4+4\ell_4}^{\pm i}/\M_{1,4+4\ell_4}})} str_{h_4} H^\ast(\M^{\pm i}_{1,4+4\ell_4},W_{4,\ell_4}) \cdot str_{h_4} H^\ast (\pi_4^{-1}(\{C_4\}), (V^\pm_{4,\ell_4})) \\
=& \dsum_{\ell_4 \geq 0, \pm}  \dfrac1{c_0(str_{h_4} N^\ast_{\M_{1,4}^{\pm i}/\M_{1,4}})} \cdot \dfrac14 \cdot str_{h_4} H^\ast (\pi_4^{-1}(\{C_4\}), (V^\pm_{4,\ell_4})),
\end{array}\]
where 
\[\begin{array}{ll}
W_{4,\ell_4} & =  \dfrac1{4 \cdot 4^{\ell_4}\ell_4 !} \dprod_{i=1}^{\ell_4} (L_{4i+1} \cdots L_{4i+4}-1);\\
V^\pm_{4,\ell_4} &= \dsum_{\bar{\ell}_4,d} \frac{Q^d}{4^{\bar{\ell}_4}\bar{\ell}_4!} \left( ev_1^\ast \bar{x}_{1,\pm i}ev_2^\ast \bar{x}_{1,\pm i}ev_3^\ast \bar{x}_{2,-1}ev_4^\ast \bar{x}_{2,-1} \dprod_{i=1}^{\ell_4} (ev_{4i+1}^\ast \bar{t}_{4,1} \cdots ev_{4i+4}^\ast \bar{t}_{4,1}) \cdot (\bar{ft}_4)_\ast \mathcal{T}\right).
\end{array}\]

Computing with the following base curve,

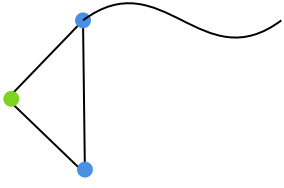
\begin{tikzpicture}[x=0.75pt,y=0.75pt,yscale=-1,xscale=1]
\draw   (86.68,25.85) -- (87.65,103.15) -- (48.52,65.46) -- cycle ;
\draw  [color={rgb, 255:red, 74; green, 144; blue, 226 }  ,draw opacity=1 ][line width=3]  (85.65,101.15) .. controls (85.65,100.05) and (86.55,99.15) .. (87.65,99.15) .. controls (88.76,99.15) and (89.65,100.05) .. (89.65,101.15) .. controls (89.65,102.26) and (88.76,103.15) .. (87.65,103.15) .. controls (86.55,103.15) and (85.65,102.26) .. (85.65,101.15) -- cycle ;
\draw  [color={rgb, 255:red, 74; green, 144; blue, 226 }  ,draw opacity=1 ][line width=3]  (84.68,25.85) .. controls (84.68,24.74) and (85.58,23.85) .. (86.68,23.85) .. controls (87.79,23.85) and (88.68,24.74) .. (88.68,25.85) .. controls (88.68,26.95) and (87.79,27.85) .. (86.68,27.85) .. controls (85.58,27.85) and (84.68,26.95) .. (84.68,25.85) -- cycle ;
\draw  [color={rgb, 255:red, 126; green, 211; blue, 33 }  ,draw opacity=1 ][line width=3]  (48.52,65.46) .. controls (48.52,64.35) and (49.42,63.46) .. (50.52,63.46) .. controls (51.63,63.46) and (52.52,64.35) .. (52.52,65.46) .. controls (52.52,66.56) and (51.63,67.46) .. (50.52,67.46) .. controls (49.42,67.46) and (48.52,66.56) .. (48.52,65.46) -- cycle ;
\draw    (86.68,25.85) .. controls (126.68,-4.15) and (146.68,55.85) .. (186.68,25.85) ;
\end{tikzpicture}

\noindent we find
\[\sum_{\ell_4} str_{h_4} H^\ast (\pi_4^{-1}(\{C_4\}), V^\pm_{4,\ell_4}) = \<\!\<\bar{y}_{1,\pm i}, \bar{x}_{1,\pm i}, \bar{t}_{2,-1}\>\!\>_{1,2_1+1_2}.\]
By writing it as a residue, their sum becomes
\[\begin{array}{ll}
&Res_{\pm i} \dfrac14\<\!\<\dfrac{\y_2(q)}{1-q^{-1}L}, \bar{\x}_1(q), \bar{\x}_2(q)\>\!\>_{1,2_1+1_2} \dfrac{dq}q \\
=& Res_{0,\infty} \dfrac14 \<\!\<\dfrac{\y_4(x^{-1})}{1-x\bar{L}}, \bar{\x}_1(x^{-1}), \bar{\x}_2(x^{-2})\>\!\>_{1,2_1+1_2} \dfrac{dx}x.
\end{array}\]

Similarly, in cases $4\pm, 5\pm$, the total contributions are

\[\begin{array}{ll}
&Res_{0,\infty} \dfrac16 \<\!\<\dfrac{\y_3(x^{-1})}{1-x\bar{L}}, \bar{\x}_1(x^{-1}),\bar{\x}_1(x^{-1})\>\!\>_{1,3_1} \dfrac{dx}x, \\
&Res_{0,\infty} \dfrac16 \<\!\<\dfrac{\y_6(x^{-1})}{1-x\bar{L}}, \bar{\x}_2(x^{-2}), \bar{\x}_3(x^{-3}) \>\!\>_{1,1_1+1_2+1_3} \dfrac{dx}x,
\end{array}\]
respectively.
This proves Theorem 1.

\author{Dun Tang, Department of Mathematics, University of California, Berkeley, 
1006 Evans Hall, University Drive, Berkeley, CA 94720.
E-mail address: dun\_tang@math.berkeley.edu}
\end{document}